\documentclass[11pt]{amsart}
\usepackage{latexsym}
\usepackage{amsfonts,amssymb,amsmath,amsthm,amscd}
\usepackage[mathscr]{eucal}
\usepackage{graphicx}

\newcommand{\F}{\mbox{${\mathcal F}$}}
\newcommand{\E}{\mbox{${\mathcal E}$}}

\newcommand{\g}{\mbox{${\mathfrak g}$}}
\newcommand{\h}{\mbox{${\mathfrak h}$}}
\newcommand{\kf}{\mbox{${\mathfrak k}$}}

\newcommand{\p}{\mbox{${\mathfrak p}$}}
\newcommand{\q}{\mbox{${\mathfrak q}$}}

\newcommand{\s}{\mbox{${\mathfrak s}$}}

\newcommand{\HH}{\mbox{${\mathbb H}$}}
\newcommand{\I}{\mbox{${\mathbb I}$}}

\newcommand{\PP}{\mbox{${\mathbb P}$}}

\newcommand{\R}{\mbox{${\mathbb R}$}}

\newcommand{\Ha}{\mbox{${\mathcal H}$}}
\newcommand{\Ly}{\mbox{${\mathcal L}$}}
\newcommand{\Qa}{\mbox{${\mathcal Q}$}}

\newcommand{\tr}{{\rm tr}}
\newcommand{\ric}{{\rm Ric}}
\newcommand{\eop}{\mbox{$\Box$}}

\makeatletter
\def\numberwithin#1#2{\@ifundefined{c@#1}{\@nocnterrr}{%
  \@ifundefined{c@#2}{\@nocnterr}{%
  \@addtoreset{#1}{#2}%
  \toks@\expandafter\expandafter\expandafter{\csname the#1\endcsname}%
  \expandafter\xdef\csname the#1\endcsname
    {\expandafter\noexpand\csname the#2\endcsname
     .\the\toks@}}}}
\makeatother
\numberwithin{equation}{section}
\newtheorem{thm}[equation]{Theorem}

\newtheorem{prop}[equation]{Proposition}

\newtheorem{ex}[equation]{Example}
\newenvironment{example}{\begin{ex} \em}{\end{ex}}
\newtheorem{rem}[equation]{Remark}
\newenvironment{rmk}{\begin{rem} \em}{\end{rem}}
\begin{document}

\title{Cohomogeneity One Shrinking Ricci Solitons: \\ An Analytic and
     Numerical Study}
\author{Andrew S. Dancer}
\address{Jesus College, Oxford University, OX1 3DW, United Kingdom}
\email{dancer@maths.ox.ac.uk}
\author{Stuart J. Hall}
\address{Department of Mathematics, Imperial College, London, SW7 2AZ,
United Kingdom}
\email{stuart.hall06@imperial.ac.uk}
\author{McKenzie Y. Wang}
\address{Department of Mathematics and Statistics, McMaster
University, Hamilton, Ontario, L8S 4K1, Canada}
\email{wang@mcmaster.ca}

\date{revised \today}

\begin{abstract}
  We use analytical and numerical methods to investigate the equations
  for cohomogeneity one shrinking gradient Ricci solitons. We show the
  existence of a winding number for this system around the subvariety of phase space
  corresponding to Einstein solutions and obtain some estimates for
  it. We prove a non-existence result for certain orbit types,
  analogous to that of B\"ohm in the Einstein case. We also carry out
numerical investigations for selected orbit types.
\end{abstract}

\maketitle

\bigskip
\setcounter{section}{-1}

\section{\bf Introduction}

A Ricci soliton consists of a complete Riemannian metric $g$ and a
complete vector field $X$ on a manifold which satisfy the equation:
\begin{equation} \label{LieRS}
{\ric}(g) + \frac{1}{2} \,{\sf L}_{X} g + \frac{\epsilon}{2} \, g = 0
\end{equation}
where $\epsilon$ is a real constant and $\sf L$ denotes the Lie derivative.
The soliton is called {\em steady} if $\epsilon =0$, {\em expanding}
if $\epsilon > 0,$ and {\em shrinking} if $\epsilon < 0$.
If $X$ is a Killing vector field, the metric is actually Einstein,
and the soliton is called {\em trivial}: we shall
exclude these in what follows. For general background, see \cite{Cetc}
and \cite{Ca2}.

As well as being natural generalisations of Einstein
metrics, Ricci solitons are of interest as generating
particularly simple solutions to the Ricci flow, for they evolve
under this flow just by its natural symmetries, that is, by diffeomorphisms
and homothetic rescalings. Ricci solitons are of crucial importance
in analysing, via dilations, singularities of the flow \cite{Pe}.
The Ricci flow is also of physical interest
as an approximation to the renormalisation group flow for nonlinear sigma-models
\cite{F}. Indeed one of the earliest nontrivial Ricci solitons,
the Hamilton cigar \cite{Ha}, was also discovered in a physical context
by Witten \cite{Wi}.

Most known examples of Ricci solitons are of {\em gradient type}, that is,
we have $X = {\rm grad} \; u$ for a smooth function $u$. Equation (\ref{LieRS})
then becomes
\begin{equation} \label{gradRS}
{\ric}(g) + {\rm Hess}(u) + \frac{\epsilon}{2} \, g = 0.
\end{equation}
We shall use the abbreviation GRS for gradient Ricci solitons.
In this case, the completeness of ${\rm grad}\ u$ actually follows from the
completeness of $g$, by the recent work of Z. Zhang \cite{Zh}.

Noncompact complete solitons may be steady, expanding or shrinking, but
nontrivial compact solitons must be shrinking and of gradient type \cite{Pe}.
In all cases K\"ahler examples have been constructed (see \cite{Koi}, \cite{Ca1},
\cite{ACGT}, \cite{FIK}, \cite{DW1}, \cite{FuW} for example, among many other authors).
In the expanding and steady cases  non-K\"ahler examples were
constructed in \cite{DW2} , \cite{DW3}, generalising
examples of Bryant \cite{Bry}, Ivey \cite{Iv}, Gastel and Kronz \cite{GK}.
There are also many homogeneous non-K\"ahler expanders (necessarily of
non-gradient type), constructed by Lauret \cite{La} and others.

In contrast, so far the only known complete nontrivial shrinking solitons
are K\"ahler, or else are rigid (i.e. obtained from a product of Einstein
spaces and a Gaussian soliton \cite{PW}).
In this paper we begin a study of shrinking solitons which are not necessarily
of K\"ahler type. We shall focus on solitons of cohomogeneity
one, when the existence of a symmetry group with hypersurface principal orbits
(or a suitable bundle ansatz) reduces the equations to a system of ordinary
differential equations. This is a natural place to start, as  homogeneous
compact solitons are trivial. Moreover, the great majority of known complete
nontrivial solitons are either homogeneous (in the case of Lauret's noncompact
expanders) or obtained by a
 cohomogeneity one/bundle-type construction.
(The main exceptions we are aware of are the
existence result \cite{WZh} for K\"ahler-Ricci solitons on toric varieties,
and the generalisation \cite{PS} to Fano toric bundles over a
generalised flag variety).

In \S 1,2 we show that the general cohomogeneity one shrinking soliton
equations yield a winding number, representing the winding of the flow
around the subvariety of phase space corresponding to Einstein
solutions. We obtain some estimates for the winding. We also introduce
a functional which is monotone in a certain region of phase space. In
\S 3 we prove a nonexistence result for compact cohomogeneity one
solitons with certain orbit types, analogous to a result of B\"ohm in
the Einstein case \cite{B3}. In \S 4  we focus on solitons of
multiple warped product type. In \S 5 we report on numerical
investigations for some orbit types, including those where solutions
were found in the Einstein case \cite{B1}. The numerics indicate that
the shrinking soliton equations on compact spaces exhibit a high degree of
rigidity; in particular we discuss several spaces where Einstein
metrics exist \cite{B1} but where numerical searches did not produce any evidence of nontrivial solitons.

\medskip
{\em Acknowledgements:} The second author was supported by an EPSRC
doctoral grant during the period when some of this research was carried out.
He would like to thank his PhD supervisor, Prof. Simon Donaldson, for his
encouragement. The third author is partially supported by NSERC Grant
No. OPG0009421

\section{\bf Cohomogeneity one Ricci soliton equations and a general winding number}

Let us consider a cohomogeneity one gradient Ricci soliton:
this means that the soliton is invariant under the action of
a Lie group $G$ with generic orbit (called the principal orbit type)
of real codimension one. We write the principal
orbit as $G/K$, where $G$ is a compact Lie group and $K$ is a closed subgroup.
We refer to \cite{DW1} for background on the soliton equations in the cohomogeneity
one setting.

We denote by $n$ the dimension of $G/K$, so the cohomogeneity one manifold $M$
has dimension $n+1$. The open and dense subset $M_0$ of the manifold consisting of
the principal orbits may be written as $I \times (G/K)$, where
$I$ is an open interval in $\mathbb R$ and  the product structure results from
choosing a geodesic which intersects all generic orbits orthogonally.
On $M_0$ we may then write the soliton potential as $u(t)$ and the metric as
\begin{equation} \label{metric}
\bar{g}:=dt^2 + g_t
\end{equation}
where, for each $t$, $g_t$ is a $G$-invariant metric on $G/K$.
Letting $\p$ be an Ad$_{K}$-invariant complement to $\kf$ in $\g$, we know
that $G$-invariant metrics on $G/K$ correspond to Ad$_{K}$-invariant
inner products on $\p$.

The cohomogeneity one gradient Ricci soliton equations on $I \times (G/K)$
are now
\begin{eqnarray}
r_t -\dot{L} + (\dot{u} - {\rm tr} L) L + \frac{\epsilon}{2} \I &=& 0
\label{soleqS}, \\
-{\rm tr} (\dot{L}) - {\rm tr} (L^2) + \ddot{u} + \frac{\epsilon}{2}&=& 0,
\label{soleqN} \\
d({\tr L}) + \delta^{\nabla} L &=& 0,
\label{soleq-mixed}
\end{eqnarray}
where $L$ denotes the shape operator of the hypersurface orbits, viewed as an endomorphism
and related to the metric by $\dot{g_t} = 2g_t L$. Also $r_t$ denotes the
Ricci endomorphism, and $\delta^{\nabla}$ is the codifferential for
$T^*(G/K)$-valued $1$-forms.
Note that the first equation represents the components of the equations
tangent to $G/K$, the second is the equation in the normal direction, and
the last equation represents the mixed directions.

The cohomogeneity one manifold can have up to two special orbits, corresponding
to the endpoints of the interval $I$. In this paper we will assume that there is
at least one special orbit, placed at $t=0$. We may then assume that
$I=(0, T)$ where $T$ is an extended positive real number. If $T< +\infty$, $M$ is
compact and additional conditions must be imposed at $0$ and $T$ to ensure that
the metric $\bar{g}$ and potential $u$ extend smoothly to all of $M$. When there
is only one special orbit, the cohomogeneity one manifold is non-compact, and
there are further conditions which ensure that $\bar{g}$ is complete.
(See \cite{EW} and \cite{DW1} for more details of the smoothness conditions.)
Here we just like to highlight one result from \cite{DW1} (cf Prop 3.19) which has
practical consequences for the construction of gradient Ricci solitons.

Assume that there is a special orbit of dimension strictly smaller
than $n$. Then if (\ref{soleqS}) holds for a sufficiently smooth ($C^3$) metric and
potential, then (\ref{soleq-mixed}) automatically holds. Furthermore, if
the well-known {\em conservation law}
\begin{equation} \label{cons1}
\ddot{u} +  ( - \dot{u} +{\rm tr}\, L )\dot{u} -\epsilon u = C
\end{equation}
holds for some fixed constant $C$, then (\ref{soleqN}) also holds. In other words,
in constructing cohomogeneity one gradient Ricci solitons, it is enough to produce a
solution to (\ref{soleqS}) satisfying the conservation law which is smooth up to order $3$.
Note that in \cite{EW} examples of principal orbit and special orbits are given
for which the smoothness conditions for $\bar{g}$ have to be checked to arbitrarily high
orders.

\bigskip
Using (\ref{soleqN}) and the trace of (\ref{soleqS}), the conservation law above
can be written in the form
\begin{equation} \label{ham}
S + \tr(L^2)- (-\dot{u}+ \tr \,L)^2 + (n-1)\frac{\epsilon}{2} = C + \epsilon u,
\end{equation}
where $S$ is the scalar curvature of the principal orbits and the constant $C$
is the same one as in (\ref{cons1}). We note that this form of the conservation
law may be viewed as the zero energy condition for a Hamiltonian constructed out of
Perelman's energy and entropy functionals.

\bigskip
Two quantities appearing in (\ref{ham}) are of interest to us.

The first is
\begin{equation} \label{xi}
\xi:=-\dot{u} + \tr \,L,
\end{equation}
which is the analogue of the mean curvature $\tr L$ of the principal orbits
for the dilaton volume element $ e^{-u} d\mu_g$.  By (\ref{soleqN}), $\xi$ is
strictly decreasing in the shrinking case ($\epsilon < 0$). Indeed we have
the useful inequality
\begin{equation} \label{xidot}
\dot{\xi} \leq \frac{\epsilon}{2} < 0.
\end{equation}
Furthermore, when there is a special orbit with codimension $k+1 \geq 2$ placed
at $t=0$, the smoothness conditions imply that $\dot{u}(0)=0$ and
${\tr \,L} = \frac{k}{t} + O(t)$ as $t \rightarrow 0$. In the complete, noncompact
case, the two facts above show that $\xi$ tends to $+\infty$ as $t$ tends to $0$ and
$\xi$ tends to $-\infty$ as $t$ tends to $+\infty$. In the compact case, an analogous
calculation to that above shows that $\xi$ tends to $-\infty$ as we approach the
second special orbit at $t=T$. Hence in the shrinking case, if the metric is
complete, $\xi$ must decrease  from $+\infty$ to $-\infty$ regardless
of whether $M$ is compact or non-compact.

Denote by $t_*$ the unique zero of $\xi$,
and introduce, as in \cite{DW2}, \cite{DW3}, the variable
\begin{equation}  \label{def-W}
W := \xi^{-1} = \frac{1}{-\dot{u} + {\tr L}}
\end{equation}
as well as a new independent variable $s$ defined by
\begin{equation} \label{st}
\frac{d}{ds} := W \ \frac{d}{dt} =
\frac{1}{(-\dot{u} + {\rm tr} L)} \ \frac{d}{dt}.
\end{equation}
We shall use a prime ${ }^{\prime}$ to denote differentiation with respect to $s$.
It is natural to let $s=-\infty$ correspond to $t=0$. The zero $t_*$ of $\xi$
corresponds to the blow-up time of $W$, which occurs at a {\em finite} value $s_*$ of $s$.
To see this, note that equation (\ref{soleqN}) can be rewritten as
$$\dot{\xi} = -\tr(L^2) + \frac{\epsilon}{2},$$
or in terms of $W$ as
\begin{equation} \label{Weqn}
      W^{\prime} = W^3 \left(\tr(L^2) - \frac{\epsilon}{2}\right).
\end{equation}
Since $\epsilon < 0$, we have $W^{\prime} \geq -\frac{\epsilon}{2} W^3$, which shows
that $W$ blows up in finite time.

We shall call a point where $W$ blows up (or equivalently, where the generalised
mean curvature $\xi$ vanishes), a {\em turning point}. For the special case of
Einstein trajectories, where $u$ is constant, this reduces to B\"ohm's notion of
turning point, that is, a point where the hypersurface becomes minimal \cite{B1}.
Note however that a turning point is not a singularity of the metric, just a
coordinate singularity. One might hope to construct compact shrinkers by glueing
together two trajectories at the turning point if the right matching conditions held.

\medskip

The second interesting quantity is
\begin{equation} \label{defE}
{\mathcal E} = C + \epsilon u,
\end{equation}
which we view both as a smooth function on $M$ and on the interval (orbit space) $I$.
Since one may add an arbitrary constant to the soliton potential, (\ref{cons1})
shows that this is compensated by a corresponding change in the constant $C$
appearing in the conservation law. This is one of the reasons for interest in the
quantity $\E$ rather than the potential. It is then natural to introduce
also ${\mathcal F} := \dot{u}$,
as $u$ is defined up to an arbitrary constant. Note that (\ref{cons1}) may now be
written as
\begin{equation} \label{eqnE}
\ddot{\mathcal E} + \xi \dot{\mathcal E} - \epsilon {\mathcal E} = 0.
\end{equation}

We now switch to derivatives with respect to $s$. We have
$$ {\mathcal E}^{\prime} = \epsilon u^{\prime} = \epsilon W {\mathcal F}.$$
Differentiating $\mathcal F$ we obtain
\begin{eqnarray*}
 {\mathcal F}^{\prime} &=& W\ \ddot{u} \\
                &=& W\left(C+ \epsilon u - \frac{\dot{u}}{W}\right) \\
                &=& W\left({\mathcal E} - \frac{\mathcal F}{W}\right)
\end{eqnarray*}
where we used (\ref{cons1}) to get the second equality above. In other words, we have
the associated system
\begin{equation} \label{EF}
\left( \begin{array}{c}
{\mathcal E} \\
{\mathcal F}
\end{array} \right)^{\prime} =
\left( \begin{array}{cc}
0 & \epsilon W \\
W & -1
\end{array} \right)
\left(\begin{array}{c}
{\mathcal E} \\
{\mathcal F}
\end{array}
\right).
\end{equation}
along the solution trajectories as long as the independent variable $s$ makes sense.

Note that the eigenvalues of the matrix are
\[
\frac{-1 \pm \sqrt{1 + 4 \epsilon W^2}}{2}
\]
so in the shrinking case ($\epsilon < 0$)
they form a complex conjugate pair with negative real part,
 provided $W > \frac{1}{2 \sqrt{-\epsilon}} $.

\medskip

We will now examine the system (\ref{EF}) more closely.

\begin{rmk}
The origin in the $(\E, \F)$-plane
(which is invariant under the flow) represents trajectories
with constant soliton potential $u$, that is,
Einstein trajectories.
\end{rmk}

\begin{rmk}
In the expanding case $\epsilon$ is positive, and hence the quadrants
$\E, \F  < 0$ and $\E , \F >0$ are invariant under the flow (cf \cite{DW3}).
In \cite{DW3} this fact was used to establish bounds on the flow
for the case of expanding solitons of multiple warped product type.

In the shrinking case this is no longer the case. Indeed, $\E^{\prime}$
has the opposite sign to $\F$, while $\F^{\prime}$ takes the same sign as $\E$
on $\F =0$.
The phase plane diagram shows that no quadrant is invariant and indeed
the flow  exhibits rotational behaviour
 around the origin in the $(\E, \F)$-plane.
\end{rmk}

\medskip
In order to analyse the behaviour of the variables $\E$ and $\F$
in the shrinking case, it is convenient to take $\epsilon = -1$.
We observe, from the equations (\ref{EF}), that
\[
(\E^2 + \F^2)^{\prime} = - 2 \F^2
\]
so $\E^2 + \F^2$ is monotonic decreasing. (Note that for a non-trivial
soliton, $\F=\dot{u}$ cannot be $0$ along the flow.)

We also calculate
\begin{equation} \label{arctanEF}
\frac{d}{ds} \tan^{-1} (\E / \F) = \frac{\E^{\prime}\F -\F^{\prime}\E}
{\E^2 + \F^2} = -W + \frac{\E \F}{\E^2 + \F^2}.
\end{equation}
If we let $\theta = \tan^{-1} (\E / \F)$, this equation becomes
\begin{equation} \label{theta}
\frac{d \theta}{ds} = - W + \frac{1}{2} \sin (2 \theta).
\end{equation}

We are interested in estimating the winding angle around the origin
of the flow in the $(\E, \F)$-plane up to the (unique) turning point
$s_*$ (point where $W$ blows up, cf the above discussion).

Let us take the $\F$-axis as horizontal and the $\E$-axis as vertical.
We shall consider trajectories emanating from a point on the positive $\E$-axis
(i.e., $\theta = \frac{\pi}{2}$) since at a special orbit we need $\dot{u} =0$,
that is, $\F=0$. Our phase-plane diagram shows that the flow will initially move into
the quadrant with $\E, \F >0$ and can never re-cross the positive
$\E$-axis from the right. The winding angle is therefore nonpositive. Analogous
statements hold for trajectories emanating from a point on the negative
$\E$-axis, where $\theta= \frac{3\pi}{2}$.

Observe further that the line $\theta = \frac{\pi}{4}$ is a reflecting barrier
for the flow unless $W \geq \frac{1}{2}$. Also note that once $W$ exceeds
$\frac{1}{2}$ then $\theta$ is monotonic decreasing, i.e., the flow
winds clockwise around the origin.

The winding angle is
\[
\int_{-\infty}^{s_*} \theta^{\prime} \; ds = \int_{-\infty}^{W^{-1}(\frac{1}{2})}
\theta^{\prime} \; ds +  \int_{W^{-1}(\frac{1}{2})}^{s_*} \theta^{\prime} \; ds,
\]
where we have arranged for $s=-\infty$ to correspond to $t=0$ (in view of (\ref{st})
and the initial conditions). The value of the first integral is between zero and
$-\frac{\pi}{4}$, from our above remarks, as until $W$ reaches $\frac{1}{2}$ the flow
is trapped between $\theta = \frac{\pi}{2}$ and $\theta = \frac{\pi}{4}$.

For the second integral, we can change variables to rewrite it
(using (\ref{Weqn}) with ${\mathcal G}:= W^2 \tr(L^2)$ and (\ref{theta})) as
\[
\int_{\frac{1}{2}}^{\infty} -\frac{1}{{\mathcal G} + \frac{1}{2} W^2} \; dW
+
\int_{\frac{1}{2}}^{\infty} \frac{\frac{1}{2} \sin (2 \theta)}
{W({\mathcal G} + \frac{1}{2} W^2)} \; dW
\]

These integrals are bounded in absolute value by $4$ and $2$ respectively, just
using the nonegativity of ${\mathcal G}$. We have therefore deduced

\begin{prop} \label{windingbound} For trajectories of the associated
  flow in the $(\E,\F)$-plane starting from either the positive
  ${\mathcal E}$ or negative $\E$ axis, the winding angle around the
  origin up to the turning point
is finite, nonpositive, and bounded below by $-(6 +
  \frac{\pi}{4})$. \eop
\end{prop}

Of course one would like solutions with winding angle  $-\pi$, so as to get
$\F=0$ at the turning point $s_*$, which is one of the conditions
needed to get a symmetric solution. The above argument shows that one cannot
achieve this by mimicking the B\"ohm techniques--we cannot make the winding
angle arbitrarily large by varying some parameter.

\begin{rmk} \label{HQ}
We can relate $\mathcal{E}, \mathcal{F}$ to the quantities used in the paper
on expanding solitons \cite{DW3} as follows. Let us define
\begin{equation}
{\mathcal H} :=  W \ \tr L \ \ \mbox{\rm and} \ \ \Qa := W^2 \E.
\end{equation}
We recall from \cite{DW3} the calculation:
\begin{equation} \label{udot}
u^{\prime}=\frac{\dot{u}}{-\dot{u} + \tr L} = W {\rm tr} L - 1 = \Ha -1
\end{equation}
which is actually valid for general orbit types. It follows that
${\mathcal F} = \frac{\Ha-1}{W}.$

We can then derive easily the equations:
\begin{equation} \label{eqnH}
\left({\mathcal H} -1 \right)^{\prime} =
\left( {\mathcal G} -1 -\frac{\epsilon}{2} W^2 \right)
({\mathcal H}-1 ) + \Qa,
\end{equation}
\begin{equation} \label{eqnQ}
\Qa^{\prime} =\epsilon W^2 \left( {\mathcal H} -1 \right) +
2 \left({\mathcal G} - \frac{\epsilon}{2} W^2 \right) \Qa,
\end{equation}
which we saw in \cite{DW3} in the special case where $\epsilon$
was positive and the cohomogeneity one metric was a warped product on $m$
factors.

The associated system (\ref{EF}) described earlier may be regarded as a
renormalised version of the above pair of equations where we replace
$\Ha -1$ and $\Qa$ by their quotients by appropriate powers of the variable $W$.
Note that the subset $\{{\mathcal H} =1, {\mathcal Q} = 0 \}$ is invariant under
the flow of the GRS equations (at least while the variable $s$ is defined), and
consists of Einstein trajectories.
\end{rmk}

\section{\bf Some properties of complete shrinkers of cohomogeneity one}

In this section we describe some properties of complete shrinking gradient Ricci
solitons with cohomogeneity one.

It is known that the scalar curvature of a complete steady or shrinking
gradient Ricci soliton must be non-negative. In the compact case, steady
GRS are in fact Ricci flat and shrinking GRS must have positive scalar
curvature (cf Proposition 1.13 of \cite{Cetc}). For the non-compact case,
the non-negativity follows from a theorem of B. L. Chen (Corollary 2.5 of
\cite{Ch}) on
complete ancient solutions of the Ricci flow. (The result was previously
known assuming sectional curvature bounds.)

For cohomogeneity one gradient Ricci solitons, the ambient scalar curvature
$\overline R$ is given by
\begin{equation} \label{ambientR}
\overline{R} = -2\ \tr(\dot{L}) - \tr(L^2) - (\tr L)^2 + S,
\end{equation}
where $S$ denotes the scalar curvature of the principal orbits (cf \S 1 of \cite{DW1}).
Now we have
\begin{eqnarray*}
 \overline{R} &=& -S + (\tr L)^2 - \tr(L^2) -2 \dot{u} \ \tr L - \epsilon n \\
   &=& -C -\epsilon u - {\dot{u}}^2 - \frac{\epsilon}{2} (n+1)  \\
   &=& -\E -\frac{\dot{\E}^2}{\epsilon^2} - \frac{\epsilon}{2}(n+1) ,
\end{eqnarray*}
where we used (\ref{soleqS}) and  (\ref{ham}) respectively
in the first two equalities. Now $\overline R \geq 0$ immediately
leads to the upper bound
\begin{equation} \label{Ebound}
     \E \leq -\frac{\epsilon}{2} (n+1),
\end{equation}
which is strict in the compact shrinking case.

We take this opportunity to emphasize that $\mathcal E$ can be viewed
both as a function on $M$ and a function on $I$, and we will frequently
switch from one viewpoint to the other in the following.

\begin{rmk} \label{steadyC}
For the steady case, the above inequality becomes $C \leq 0$ with equality
iff $\dot{u}$ is identically zero, i.e., the soliton is trivial.
\end{rmk}

\begin{rmk} \label{expanderC}
In the case of a complete non-trivial expanding GRS of cohomogeneity one,
Corollary 2.3 of \cite{Ch}, applied to the Ricci flow generated by the expander,
implies that $\overline R > -\frac{\epsilon}{2}(n+1).$ (Note that our ambient
manifold has dimension $n+1$.) Since $\epsilon$ is now positive,
the trace of (\ref{gradRS}) yields $\Delta \E < 0$, where $\Delta$
is the negative Laplacian of the cohomogeneity one metric. As $\E$ is a function of
$t$ only, $\Delta \E = \ddot{\E} + (\tr L)\dot{\E}$. Using the above facts in
(\ref{eqnE}) we obtain $\E < 0$. In particular, if we assume without loss of generality
that $u(0)=0$, then the constant $C$ in the conservation law must be negative.
\end{rmk}

Going back to the shrinking case, recall that the smoothness conditions at a special
orbit placed at $t=0$ imply that
\begin{equation} \label{expand-u}
 u = u(0) + \frac{1}{2} \ \ddot{u}(0) \ t^2 + \cdots
\end{equation}
and
\begin{equation} \label{expand-trL}
 \tr L = \frac{k}{t} + \cdots
\end{equation}
for sufficiently small values of $t$. (Here the mean curvature
$\tr L$ of the principal orbits is taken with respect to the ``outward" normal
$\frac{\partial}{\partial t}$.)
Hence, in the case of a complete shrinking GRS of cohomogeneity one, the conservation
law (\ref{cons1}) and (\ref{Ebound}) give
\begin{equation} \label{udotdot}
 (k+1) \ddot{u}(0) = C + \epsilon u(0) = {\mathcal E}(t=0) \leq  -\frac{\epsilon}{2} (n+1).
\end{equation}
So for the initial value $\ddot{u}(0)$ we have the upper bound
\begin{equation} \label{iv-u}
\ddot{u}(0) \leq -\frac{\epsilon}{2}\left(\frac{n+1}{k+1} \right),
\end{equation}
which must be strict at both endpoints in the case of compact shrinkers.
This means that the trajectories of the associated flow in the $(\E, \F)$-plane
cannot start arbitrarily high on the positive $\E$-axis if they correspond to
complete solutions, compact or not. By contrast, note that there is no
such restriction on $\ddot{u}(0)$ from the initial value problem (cf \cite{Bu})
around special orbits.

Note also that the upper bound $\E(t=0) \leq -\frac{\epsilon}{2}(n+1)$ is attained by the
conical Gaussian soliton on $\R^{n+1}$. Indeed, in this case $\E = -\frac{\epsilon}{2}(n+1)
-\frac{\epsilon^2}{4}t^2$ (see Remark \ref{CG}).

\begin{prop}  \label{compactE}
Let $(M, \bar{g})$ be a non-trivial compact GRS of cohomogeneity one
with $G$-invariant potential function $u$. Let $[0, T]$ denote the closure of
the interval $I$ and ${\mathcal E} = C + \epsilon u$, where the constant $C$
is that appearing in $($\ref{ham}$)$. Then:
\begin{enumerate}
\item[a.] $\mathcal E$ must change sign and is a Morse-Bott function on $M$.
\item[b.] $\mathcal E < -\frac{\epsilon}{2}(n+1)$.
\item[c.] $\mathcal E$ has at most $4$ critical points in $(0, T)$. If $t_0 \in (0, T)$
     is a critical point then it is either a local maximum if ${\mathcal E}(t_0) > 0$
     or a local minimum if ${\mathcal E}(t_0) < 0$.
\end{enumerate}
\end{prop}
\begin{proof}
We note first that by (\ref{eqnE}) $\mathcal E$ can be viewed as a solution
to the {\em linear} equation
$$\ddot{f} + \xi \dot{f} - \epsilon {f} = 0.$$
On $(0, T)$ the coefficient $\xi$ is finite-valued and at least $C^1$. Since the
soliton is non-trivial, $\mathcal E$ cannot vanish at an interior critical point $t_0$.
Also, $\ddot{\mathcal E}(t_0) = \epsilon {\mathcal E}(t_0)$ holds at such a point,
and the only non-vanishing component of the Hessian of $\mathcal E$
(as a function on $M$) is that in the
$\frac{\partial}{\partial t}$ direction. So the critical submanifold of $\mathcal E$
at $t_0$, a principal orbit, is non-degenerate. This equation also implies the second part of c.

Next we consider the situation at $t=0$. (The situation at $t=T$ is completely analogous.)
By smoothness, we have $\dot{\mathcal E}(0)= \epsilon \dot{u}(0) = 0$. If
${\mathcal E}(0) = 0,$ our soliton would be trivial, as the origin $({\mathcal E}, {\mathcal F})=
(0, 0)$ is invariant under the associated flow (\ref{EF}) and represents the
Einstein trajectories. Now the computation (\ref{udotdot}) shows that
$\ddot{u}(0)= \frac{1}{\epsilon}\,\ddot{\mathcal E}(0)$
has the same sign as ${\mathcal E}(0)$. In particular, the two singular orbits are
also non-degenerate critical submanifolds of $u$, so that $u$ is Morse-Bott. Since $M$
is compact, $\mathcal E$ can only have a finite number of critical points,
considered as a function on $I$. Because
$6 + \frac{\pi}{4} < \frac{5}{2} \pi$, we may apply Proposition \ref{windingbound}
separately to each of the intervals $[0, t_*]$ and $[t_*, T]$ to estimate the total
number of critical points in $(0, T)$. (Here as usual
$t_*$ denotes the unique turning point).
Indeed, when ${\mathcal E}(0)$ and ${\mathcal E}(T)$
have the same sign, the number of critical points is odd, and so is bounded above by $3$.
When ${\mathcal E}(0)$ and ${\mathcal E}(T)$ have opposite signs, then number of critical
points is even, and so is bounded above by $4$.

The inequality in part b is just (\ref{Ebound}) for the compact case.

It remains to show that $\mathcal E$ must change sign. The (negative) Laplacian of
$\mathcal E$ is
$$ \Delta {\mathcal E} = \ddot{\mathcal E} + (\tr L) \dot{\mathcal E} = \epsilon {\mathcal E} +
       \frac{1}{\epsilon} \ \dot{\mathcal E}^2, $$
where we have used (\ref{eqnE}) in the second equality above. Hence the
integral of $\mathcal E$ is negative if the soliton is non-trivial.

Suppose that ${\mathcal E} \leq 0$ on $[0, T]$. By what we already proved,
${\mathcal E}$ must then be negative at $t=0$ and $t= T$, and hence $\ddot{\mathcal E}$
must be positive at these endpoints. Therefore, $\mathcal E$ must have
interior critical points. At each such point $\mathcal E$ is negative,
so that $\ddot{\mathcal E}$ is positive. We obtain a contradiction at the first
interior critical point, which must be a local maximum because $\mathcal E$ is
convex for a little while after $t=0$.
\end{proof}

\begin{rmk} At present the only known examples of
nontrivial compact shrinking Ricci solitons
are K\"ahler. These are generalizations of the Koiso examples
(cf \cite{ACGT}, \cite{DW1} and references therein), and $\mathcal E$ is monotone
(no interior critical points) in these examples as a result of the K\"ahler condition.
In particular, $\mathcal E$ is a perfect Morse-Bott function.
\end{rmk}

In order to state the analogous result in the complete non-compact case, we need
to recall a general result of Cao and Zhou (\cite{CZ}, Theorem 1.1) concerning
the behaviour of the potential function. Adapted to the cohomogeneity one situation
and with our notation, their theorem implies that, for sufficiently large $t$,
\begin{equation} \label{caozhou}
-\frac{\epsilon}{2}(n+1) + \frac{\epsilon}{4}(t\sqrt{-\epsilon}  + c_2)^2
\leq {\mathcal E}(t)= C + \epsilon u(t) \leq -\frac{\epsilon}{2}(n+1) +
   \frac{\epsilon}{4}(t\sqrt{-\epsilon} - c_1)^2
\end{equation}
where $c_1, c_2 > 0$ are constants which depend on $\dim M = n+1$ and the geometry
of the unit ball in $M$ centred at a point on the singular orbit, e.g., the initial
point of the solution trajectory associated to the soliton.

\begin{prop} \label{noncompactE}
Let $(M, \bar{g})$ be a non-trivial complete, non-compact GRS of cohomogeneity one
with $G$-invariant potential function $u$. Let ${\mathcal E} = C + \epsilon u$, where
the constant $C$ is the one appearing in $($\ref{ham}$)$. Then:
\begin{enumerate}
\item[a.] $\mathcal E$ must change sign and is a Morse-Bott function on $M$.
\item[b.] $\mathcal E$ satisfies $($\ref{caozhou}$)$ above. In particular,
      $\mathcal E \leq -\frac{\epsilon}{2}(n+1)$ and eventually decreases
      monotonically to $-\infty$.
\item[c.] $\mathcal E$ has at most $5$ critical points in $(0, +\infty)$.
     If $t_0 \in (0, +\infty)$ is such a critical point then it is either a
     local maximum if ${\mathcal E}(t_0) > 0$
     or a local minimum if ${\mathcal E}(t_0) < 0$.
\end{enumerate}
\end{prop}
\begin{proof} The local properties of $\mathcal E$ follow as in the compact case.
The inequality (\ref{caozhou}) implies that $\mathcal E$ tends to $-\infty$. Suppose
${\mathcal E} \leq 0$ everywhere. Since ${\mathcal E}(0)< 0$ and $\mathcal E$ is
convex for a while after $t=0$, there must be a critical point at some $t_0 > 0$.
At the first such critical point we have a contradiction as in the compact case.
So $\mathcal E$ changes sign.

      Next choose $t_1$ such that ${\mathcal E} < 0$ for all $t \geq t_1$.
So any critical point $t_2 > t_1$ must have $\ddot{\mathcal E} (t_2) > 0$ and hence
must be a minimum. Since $\mathcal E$ tends to $-\infty$, we now see that
 there are actually no critical points in $[t_1, +\infty)$,
and so the number of critical points is finite. It also follows that $\dot{\mathcal E} < 0$
for $t \geq t_1$.

     To get a more precise bound on the number of interior critical points, we
will modify the arguments in \S 1 leading up to Proposition \ref{windingbound}. (Note
that here we do not need the work of Cao-Zhou.) Recall that $\xi$ has a unique
zero at $t=t_*$. For the portion of the trajectory with $t \leq t_*$, the proof of
Proposition \ref{windingbound} gives a lower bound of $-(6+\frac{\pi}{4})$ for
the winding angle. For the portion with $t \geq t_*$, we can introduce the
independent variable
\begin{equation} \label{def-sigma}
    \sigma = s_* - \int_{t_*}^{t} \ \xi(x)\ dx, \ \ \ \mbox{i.e.,} \ d\sigma = -\xi dt.
\end{equation}

    Differentiating ${\mathcal E}, {\mathcal F}$ with respect to $\sigma$ (with
$\epsilon= -1$), we now have the following analogue of (\ref{theta}):
$$ \frac{d\theta}{d\sigma} = W - \frac{1}{2} \sin(2\theta)$$
Note that the angle $\theta$ is well-defined and smooth as $t$ passes through $t_*$
and as $t \rightarrow +\infty$, so does $\sigma$.

We need to estimate the integral
$$ \int_{s_*}^{\infty} \ \ \frac{d\theta}{d\sigma} \ d\sigma.$$
Recall that $W$ is now negative and increases strictly to $0$ as $\sigma \rightarrow +\infty$
and satisfies the equation
$$ \frac{dW}{d\sigma} = -W \left(\frac{1}{2} W^2 + {\mathcal G}\right). $$
Let $\sigma_1$ be the value such that $W(\sigma_1) = -\frac{1}{2}$. As in \S 1, the
integral on $[s_*, \sigma_1]$ is bounded below by $-6$. On $[\sigma_1, +\infty)$
the net change in $\theta$ is bounded below by $-\pi$ because whatever $\theta(\sigma_1)$
equals, one encounters a reflecting barrier of the flow (in the clockwise direction)
at or before $\theta(\sigma_1) - \pi$. Hence the integral over the entire
trajectory is bounded below by $-(6+ \frac{\pi}{4}) - (6 + \pi) > -(5.1)\pi.$
This gives the desired bound for the number of critical points.
\end{proof}

\begin{rmk} So far the only examples of complete non-compact shrinking
Ricci solitons are the K\"ahler examples of \cite{FIK} and their
generalizations (cf \cite{DW1}), and (up to covers) the
Gaussian soliton and its product
with a finite number of Einstein spaces. In all these cases,
$\mathcal E$ is monotone decreasing. Critical points of the soliton potential
for complete non-compact shrinkers have been studied in \cite{Na}
(cf Corollary 2.1) assuming Ricci bounds.

\end{rmk}

Next we introduce a function which has the Lyapunov property
for the cohomogeneity one GRS equations in a region of phase space.
This generalises an analogous function due to C. B\"ohm \cite{B2} for the Einstein case.
Let
\begin{equation} \label{def-F}
{\mathscr F} := v^{\frac{2}{n}} \left(S + \tr((L^{(0)})^2)   \right)
\end{equation}
where $v(t)=\sqrt{\det g_t}$, $S$ is the scalar curvature of the principal orbit, and
$L^{(0)} := L - \frac{\tr L}{n} \I$ is the traceless part of the shape operator.
Modulo the conservation law (\ref{ham}), $\mathscr F$ can also be written as
\begin{equation} \label{def2-F}
{\mathscr F}= v^{\frac{2}{n}}\left(\frac{n-1}{n} (\tr L)^2 + \dot{u}(\dot{u} - 2\ \tr L)
       + {\mathcal E} - \frac{\epsilon}{2}(n-1)   \right).
\end{equation}

\begin{prop} \label{Lyapunov}
Along the trajectory of a cohomogeneity one GRS of class $C^3$ we have
$$ \dot{\mathscr F} = -2 v^{\frac{2}{n}} \ \tr((L^{(0)})^2) \left(\xi - \frac{1}{n} \tr L \right).$$
Hence $\mathscr F$ is non-decreasing wherever $\xi \leq \frac{1}{n} \tr L$.
\end{prop}

\begin{proof}
If we differentiate (\ref{def2-F}), use $\dot{v}= (\tr L)v$ and (\ref{soleqN}), we obtain after
some simplification
\begin{eqnarray*}
 \dot{\mathscr F} &=& \frac{2}{n}\  v^{\frac{2}{n}} (\tr L) \left( \frac{n-1}{n} (\tr L)^2 +
         \dot{u}^2 -2\dot{u} (\tr L) + {\mathcal E} - \frac{\epsilon}{2}(n-1) \right)  \\
         &    & + \ \ v^{\frac{2}{n}} (\tr L) \left(\frac{n-1}{n} \ \epsilon - \frac{2}{n}\ddot{u}
               + 2\tr(L^2)\left(\frac{\dot{u}}{\tr L} - \frac{n-1}{n}  \right)   \right).
\end{eqnarray*}
Splitting up $\tr(L^2)$ as $\tr( (L^{(0)})^2) + \frac{1}{n} (\tr L)^2$ and simplifying gives
\begin{eqnarray*}
\dot{\mathscr F} & = & v^{\frac{2}{n}} (\tr L) \left( \frac{2(n-1)}{n^2} (\tr L)^2
          -\frac{2}{n}(\ddot{u} - \dot{u}^2 + 2 \dot{u} (\tr L) - {\mathcal E})  \right. \\
         &  &  + \left.  \  2\left(\frac{\dot{u}}{\tr L} - \frac{n-1}{n}\right) \tr( (L^{(0)})^2 ) +
             \frac{2}{n} \dot{u} (\tr L) -\frac{2(n-1)}{n^2} (\tr L)^2   \right) \\
         & = & v^{\frac{2}{n}}(\tr L)\left(2\tr((L^{(0)})^2)\left(\frac{\dot{u}}{\tr L} -1 + \frac{1}{n} \right)\right) \\
         &=& -2v^{\frac{2}{n}} \tr((L^{(0)})^2) \left(\xi - \frac{1}{n} (\tr L)\right),
\end{eqnarray*}
as required. Note that we used the conservation law (\ref{cons1}) in the penultimate equality.
\end{proof}

\begin{rmk} \label{C3smooth}
The $C^3$ regularity assumption is natural since for a GRS the conservation law (\ref{cons1})
is usually derived from the equation for the Laplacian of $du$, cf Remark 3.8 in \cite{DW1}.
Conversely, as discussed in \S 1, in constructing a complete GRS, this hypothesis allows us
to work with (\ref{soleqS}) and (\ref{cons1})  as long as there is a singular orbit of
dimension strictly smaller than that of a principal orbit.
\end{rmk}

\section{\bf Non-existence results}

Fundamental results of Perelman \cite{Pe} imply that on a closed manifold all Ricci
solitons are of gradient type and all non-trivial solitons are shrinkers.
If the soliton metric is invariant under a compact group of isometries,
then by averaging over the group one obtains a new soliton potential that
is an invariant function. In particular the metric of a compact homogeneous
Ricci soliton must be Einstein. Starting from dimensions  $ \geq 12,$ examples of
compact (simply connected) homogeneous spaces without homogeneous Einstein
metrics are known, cf \cite{WZ} and \cite{BK}. Hence these homogeneous
spaces do not admit any homogeneous soliton structures, trivial or not.

We consider next the compact cohomogeneity one case. In this case, the orbit space is
either a circle or a closed interval. In the former case, the fundamental group
of the manifold is infinite, so by the work of \cite{De}, \cite{FG}
no cohomogeneity one Ricci soliton structures exist. Let us then assume that the
orbit space is a closed interval and the singular orbits have dimension strictly
smaller than that of the principal orbits. B\"ohm showed in \cite{B3} that for certain
principal orbit types $G/K$ there are no compact cohomogeneity one Einstein metrics.
His hypotheses are that there exists a $G$-invariant distribution on the principal
orbits $G/K$ such that:

\medskip
(i) the tracefree Ricci of any $G$-invariant metric on $G/K$
is negative definite on the distribution,

(ii) the distribution is irreducible,

(iii) the distribution has trivial intersection with the vertical
(collapsing) space of the two special orbits.

Many examples of such orbit types were given in \cite{B3}, \S 7.

It turns out that B\"ohm's argument also works in the Ricci soliton case,
as the extra Hessian terms $\dot{u} L$ and $\ddot{u}$ in (\ref{soleqS}), (\ref{soleqN})
do not play an essential role in the proof. For the convenience of the reader,
we give a self-contained proof for the soliton case below. As a consequence,
one concludes that the examples in \cite{B3} also do not admit any non-trivial
$G$-invariant Ricci soliton structures.

Let $\overline{M}$ be a closed manifold with a cohomogeneity one smooth
action by a compact Lie group $G$ such that the orbit space is a finite closed
interval and no singular orbit is exceptional. Let $(\bar{g}, u)$ be a gradient
Ricci soliton structure where both the metric and potential are $G$-invariant.
We may choose a unit speed geodesic $\gamma$ that intersects all principal orbits
orthogonally. Then there are closed subgroups $K \subset H_j \subset G, \ j=1, 2$,
such that $G/K$ is the principal orbit type along the interior of $\gamma$ and $G/H_j$
are the singular orbits at the endpoints of $\gamma$. Assume that the domain of
$\gamma$ is $[0, \tau]$, which may be identified with the orbit space of the
cohomogeneity one $G$-action.

Let us choose an ${\rm Ad}_K$-invariant decomposition $\g = \kf \oplus \p$.
Note that in this section we are not making any assumptions on the
multiplicities of the irreducible $K$-summands in $\p$.
For each of the singular orbits $G/H_j$, $j=1, 2,$ we have a
decomposition $\p = \s_j \oplus \q_j$, where $\h_j \approx \s_j \oplus \kf,$
$\s_j$ is ${\rm Ad}_K$-invariant, and $\q_j$ is ${\rm Ad}_{H_j}$-invariant.
In order to address smoothness issues at $G/H_1$ it is convenient
to fix a $G$-invariant background metric $\beta$ on $G/K$ such that on $\s_1$ it
induces the constant curvature $1$ metric on the sphere $H_1/K$ and on $\q_1$ it
induces a fixed $G$-invariant metric on $G/H_1$. (Such a background metric
will not in general come from a bi-invariant metric on $G$.) Now on the open
subset consisting of principal orbits, one can write the
soliton metric $\bar{g}$ in the form $dt^2 + g_t$ where $g_t$ is a one-parameter
family of invariant metrics on the principal orbit $G/K,$ regarded as ${\rm Ad}_K$-
invariant endomorphisms of $\p$ which are symmetric with respect to the fixed
background metric $\beta$. On the other hand, the Ricci endomorphisms $r_t$ and
shape operators $L_t$ are symmetric with respect to $g_t$ but not necessarily
with respect to $\beta$.

We next decompose $\p$ as a sum
\begin{equation} \label{p-decomp}
\p = \p_1 \oplus \cdots \oplus \p_r
\end{equation}
where each $\p_i$ is itself a sum of {\em equivalent} ${\rm Ad}_K$-irreducible
summands, and, for all $i\neq j$, no summand of $\p_i$ is equivalent to a summand of
$\p_j$. Such a decomposition is unique up to permutation, and the summands $\p_i$ are
orthogonal with respect to $\beta$ as well as $g_t$. Note that any ${\rm Ad}_K$-invariant
$g_t$-symmetric endomorphism of $\p$ must map $\p_i$ into $\p_i$ and we will denote
by $\tr_i$ its trace on the summand $\p_i$. Its full trace on $\p$ will be denoted by
$\tr$ as usual.

Recall that the relative volume $v(t) = \sqrt{\det g_t}$ of the principal orbits
satisfies the equation $\dot{v} = (\tr L)v.$ Consider the conformally related metrics
$${\tilde g} = v^{-\frac{2}{n}} \ g. $$
One easily computes that
\begin{equation} \label{d-gtilde}
 \dot{\tilde g} = 2 {\tilde g} L^{(0)}
\end{equation}
and
\begin{equation} \label{dd-gtilde}
 \ddot{\tilde g} = \dot{\tilde g} {\tilde g}^{-1} \dot{\tilde g}  + 2 \tilde{g} \dot{L^{(0)}},
\end{equation}
where $L^{(0)}= L - (\frac{\tr L}{n}) \ \I $ is the trace-free part of $L$.

Consider the quantities
\begin{equation}  \label{defFi}
    F_i:= \frac{1}{2} \ {\tr}_i({\tilde g}^2).
\end{equation}
Using (\ref{d-gtilde}) and (\ref{dd-gtilde}) we obtain
\begin{equation} \label{d-Fi}
      \dot{F_i} = {\tr}_i ({\tilde g}{\dot{\tilde g}}) = 2 \ {\tr}_i ({\tilde g}^2  L^{(0)}),
\end{equation}
\begin{equation} \label{dd-Fi}
     \ddot{F_i} = {\tr}_i({\dot{\tilde g}}^2) + {\tr}_i({\tilde g}{\dot{\tilde g}}{\tilde g}^{-1}{\dot{\tilde g}})
           + 2\ {\tr_i}({\tilde g}^2 \dot{L^{(0)}}).
\end{equation}

To compute further, we need to break the terms in Eq. (\ref{soleqS}) into
their trace and trace-free parts. Recall that the generalised mean curvature
$\xi$  is given by $-\dot{u} + \tr L.$ Then one obtains, using (\ref{soleqS}),
$$ \dot{L^{(0)}} = r^{(0)} - \xi L^{(0)} + \left(\frac{S}{n} + \frac{\epsilon}{2} -\frac{\tr{\dot{L}}}{n}
             -\left(\frac{\tr L}{n}\right) \xi \right) \ \I $$
where $r^{(0)}$ is the trace-free part of $r$ and $S$ is its trace, i.e., the
scalar curvature. Note that the trace of Eq. (\ref{soleqS}) yields the relation
$$ S  - \tr(\dot{L}) -\xi \ \tr{L} + \frac{n\epsilon}{2} = 0.$$
Hence we obtain
\begin{equation} \label{d-tracelessL}
   \dot{L^{(0)}} + \xi L^{(0)} - r^{(0)} = 0.
\end{equation}
Substituting this into (\ref{dd-Fi}) and using (\ref{d-Fi}) we finally get
\begin{equation} \label{eqnFi}
  \ddot{F_i} + \xi \dot{F_i} = {\tr}_i({\dot{\tilde g}}^2) + {\tr}_i({\dot{\tilde g}}{\tilde g}^{-1}{\dot{\tilde g}}{\tilde g})
             + 2\ {\tr}_i({\tilde g}^2 r^{(0)}),
\end{equation}
which is an analogue of the formula in Proposition 3.2 in \cite{B3}.

We are now ready to prove

\begin{thm} \label{nonexistence}
Let $\overline{M}$ be a closed cohomogeneity one $G$-manifold as described above. Assume
that some summand $\p_{i_0}$ in $($\ref{p-decomp}$)$ is actually ${\rm Ad}_K$-irreducible
and that for any $G$-invariant metric on $G/K$, the restriction to $\p_{i_0}$
of its traceless Ricci tensor is always negative definite. Assume further that
$\p_{i_0} \cap \s_j = \{0\}$ for $j=1, 2.$

Then there cannot be any $G$-invariant gradient Ricci soliton structure on
$\overline{M}$.
\end{thm}

\begin{proof} (cf \cite{B3})
Consider the function $F_{i_0}$ given by (\ref{defFi}). We claim that as $t$ approaches
$0$, $F_{i_0}(t)$ approaches $+\infty$. This follows since $v$ tends to zero
(the singular orbit $G/H_1$ has strictly smaller dimension than the principal
orbits), and the irreducibility of $\p_{i_0}$ together with the fact that
$\p_{i_0} \cap \s_1 = \{0\}$ imply that the endomorphism $g| \p_{i_0}$ tends
to a multiple of the identity. Similarly, we see that $F_{i_0}(t)$ tends to
$+\infty$ as $t$ tends to $\tau$. (Note that $F_{i_0}$ depends on the background
metric, but the functions defined using different backgrounds differ only by
positive constants, as $\p_{i_0}$ is irreducible.)

It follows that $F_{i_0}$ has a global minimum at some point $t_* \in (0, \tau)$.
By (\ref{d-Fi}), $\dot{\tilde g}(t_*)$ restricts to zero on $\p_{i_0}$, and using
this in (\ref{eqnFi}) with the negative definiteness of $r^{(0)}$ on $\p_{i_0},$
we arrive at a contradiction.
\end{proof}

\begin{rmk} \label{nonexamples}
Examples of cohomogeneity one manifolds which satisfy the hypotheses of Theorem
\ref{nonexistence} include $S^{k+1} \times (G^{\prime}/K^{\prime}) \times M_3 \times \cdots \times M_r$
where $M_3, \cdots, M_r$ are arbitrary compact isotropy irreducible homogeneous spaces,
$G^{\prime}/K^{\prime}= {\rm SU}(\ell + m)/({\rm SO}(\ell) {\rm U}(1) {\rm U}(m))$
(a bundle over a complex Grassmannian with a symmetric space as fibre), and
$\ell \geq 32, m=1, 2, k=1, \ldots, [\ell/3]$  (cf \cite{B3}). The significance
of the spaces $G^{\prime}/K^{\prime}$ is that they do not admit any $G^{\prime}$-invariant
Einstein metrics (cf \cite{WZ}). More information about the structure of such
$G^{\prime}/K^{\prime}$ can be found in \cite{B4}, cf Theorem B in particular.
\end{rmk}

\section{\bf Multiply warped principal orbit types}

In the final section \S 5 we are going to conduct numerical
investigations of the soliton equations for certain relatively simple
orbit types, which in the Einstein case were studied by B\"ohm \cite{B1}.
In the current section we shall therefore see how the soliton equations
simplify under certain assumptions. We shall in particular make some remarks
about the case of multiple warped products, which have some special features
(some of the B\"ohm examples are warped products with two factors).

Let us first assume that the isotropy
representation is {\em multiplicity free}, that is, $\p$ decomposes into pairwise
inequivalent real $K$-modules $\p_1 \oplus \ldots \oplus \p_m$. We let $d_i$
denote the real dimension of $\p_i$ and let $n= d_1 + \ldots + d_m$. Note that
the summands $\p_i$ have a different meaning than that in \S 3.

Under the multiplicity free assumption, both the metric and shape operator
diagonalise with respect to the decomposition $\p_1 \oplus \ldots \oplus \p_m$.
We write the components as $g_i^2$ and  $L_i =\frac{\dot{g_i}}{g_i}$ respectively.
Furthermore, the Ricci term $r_t$ also diagonalises with components $r_i$.

In addition to the variable $W$ defined by (\ref{def-W}), we introduce variables
\begin{eqnarray}
X_i &=& \frac{\sqrt{d_i}}{( -\dot{u} + {\tr} L)} \frac{\dot{g_i}}{g_i} \label{def-Xi},\\
Y_i &=& \frac{\sqrt{d_i}}{g_i} \frac{1}{(-\dot{u} + \tr{L})} \label{def-Yi}.
\end{eqnarray}
Note that the non-negative quantity $\mathcal G = W^2 \tr(L^2)$ at the end of \S 1
becomes $ \sum_{i=1}^{m}  X_i^2 $ and ${\mathcal H} = W \tr L$ becomes
$\sum_{i=1}^m \sqrt{d_i} X_i$.

The gradient Ricci soliton system (\ref{soleqS})-(\ref{soleqN}) becomes the following equations
\begin{eqnarray}
X_{i}^{\prime} &=& X_i \left( \sum_{j=1}^{m} X_j^2 -1 \right) +
\sqrt{d_i} \, r_i W^2 + \frac{\epsilon}{2} \left( \sqrt{d_i}-X_i \right)
W^2 \, , \label{eqnX} \\
Y_{i}^{\prime} &=&  Y_i \left( \sum_{j=1}^{m} X_j^2 -\frac{X_i}{\sqrt{d_i}}
-\frac{\epsilon}{2}W^2 \right) \, , \label{eqnY} \\
W^{\prime} &=& W \left( \sum_{j=1}^{m} X_j^2 - \frac{\epsilon}{2}W^2  \right).
\label{eqnW}
\end{eqnarray}
Note that (\ref{soleq-mixed}) is automatically satisfied under the multiplicity free
hypothesis (cf Remark 2.21 in \cite{DW1}).

The  term involving $r_i$ in (\ref{eqnX}) can be expressed as rational
functions of the $Y_j$. Without loss of generality we can take $W$ to be positive
(before the blow-up time).

In the above variables,  the conservation law (\ref{ham}) becomes:
\begin{equation} \label{cons2}
\Ly +  \left(\frac{n-1}{2} \right)\epsilon W^2 = (C + \epsilon u)W^2 \ \ (= \E W^2)
\end{equation}
where $C$ is a constant and
\begin{equation} \label{Ly}
\Ly := \frac{{\rm tr} (L^2)}{(-\dot{u} + {\rm tr} L)^2}
+ \frac{{\rm tr}(r_t)}{(-\dot{u} + {\rm tr} L)^2}  -1.
\end{equation}
In the steady case $(\epsilon=0)$, we recall that $\Ly$ is a Lyapunov
function for the flow, but this fails for $\epsilon$ nonzero.
The first term in $\Ly$ is just ${\mathcal G}=\sum_{i=1}^{m} X_i^2$.

\bigskip
Let us now focus on the multiple warped product situation (see \cite{DW2}
for the steady and \cite{DW3} for the expanding case). Here
the hypersurface is a product of Einstein manifolds with positive Einstein
constants $\lambda_i$. Now $r_i = \frac{\lambda_i}{g_i^2}$ so the
term involving $r_i$ in the equation for $X_i$ is just $Y_i^2$ (up to
a positive multiplicative constant).

Moreover,  we have ${\rm tr}(r_t)= S > 0$ and it follows from (\ref{def-Xi}) and (\ref{def-Yi}) that
\[
\Ly = \sum_{i=1}^{m} (X_i^2 + Y_i^2) - 1,
\]
so $\Ly$ is bounded below. This can be used to give more precise
information about the flow, especially in the neighbourhood of a turning point.
By contrast, for more general orbit types, the ${\rm tr}(r_t)$
term and the functional $\Ly$ involve $\frac{Y_k^2 Y_p^2}{Y_j^2}$  terms.
Moreover $\Ly$ is indefinite.

\medskip
In the multiple warped product case we can sharpen the analysis of
the functional $\mathscr F$ that was introduced in \S 2. As we are dealing
with shrinking solitons, $\epsilon$ is negative.

\begin{prop} \label{Fincreases}
Let $\gamma(t)$ denote a solution trajectory of the GRS equations in the multiple
warped product case. Assume that the soliton is complete. Then the quantity
$\eta:= \xi - \frac{1}{n} \tr \,L$ eventually becomes and stays negative. Hence
${\mathscr F}$ eventually is non-decreasing, and is then strictly increasing
unless the shape operator $L$ becomes a multiple of the identity or we reach
a second singular orbit.
\end{prop}
\begin{proof} We shall assume that our soliton is non-compact; otherwise,
$\dot{u}$ is close to zero and $\tr \,L$ is close to $-\infty$ near $t=T$,
so that all the assertions hold by Proposition \ref{Lyapunov}.

Let us consider the region $\{\eta < 0\}$. We claim that it is invariant
under the forward flow of the GRS equations. By abuse of notation suppose
$\gamma(t)$ is a trajectory of the flow starting at $t_0$ with $\eta(\gamma(t_0))) < 0.$
Let $t_1 > t_0$ be the first time when $\gamma$ leaves the region $\{\eta< 0\}$.
Then
$$\xi(t_1) = \frac{1}{n} (\tr \,L)|_{t=t_1}.   $$
On the other hand
\begin{eqnarray*}
  \dot{\eta} &=& \dot{\xi} - \frac{1}{n} \tr(\dot{L}) \\
       &=& - \tr(L^2) - \frac{S}{n} + \frac{1}{n} \ \xi \ \tr \,L,
\end{eqnarray*}
where we have used (\ref{soleqN}) and the trace of (\ref{soleqS}). At $t=t_1$ we have
\begin{eqnarray*}
 \dot{\eta}(t_1) & = & -\frac{1}{n} S(t_1) + \left(\frac{(\tr \,L)^2}{n^2} - \tr(L^2) \right)(t_1) \\
      & \leq & -\frac{1}{n} \left(S + \frac{n-1}{n} (\tr \,L)^2 \right)(t_1) < 0  \\
\end{eqnarray*}
since $S(t_1) > 0$. This means that $\eta(\gamma(t))$ must be positive slightly to the left
of $t=t_1$, a contradiction to the choice of $t_1$.

It remains to see that $\eta(\gamma(t))$ is negative somewhere. The previous
paragraph implies that once $\eta(\gamma(t_1))$ is zero for some $t_1$, it must
be negative to the right of $t_1$. So suppose that $\eta(\gamma(t)) > 0$ for all $t$.
Then by (\ref{soleqN}),
$$ \frac{\dot{v}}{nv} = \frac{1}{n} (\tr \,L) < \xi \leq \frac{\epsilon}{2}t + a $$
for some positive constant $a$ for large values of $t$. It follows that the
cohomogeneity one manifold has finite volume, contradicting a theorem of H.-D. Cao
and X.-P. Zhu (\cite{Ca3}, Theorem 3.1).
\end{proof}

\begin{rmk} In the above proof the only time we used the hypothesis that the
hypersurface is of multiple warped product type is when we asserted $S(t_1)> 0$.
\end{rmk}

\begin{prop} \label{Fbound}
Let the hypersurface $($principal$)$ orbit be a product $(M_1, h_1) \times \dots \times
(M_m, h_m)$ of Einstein manifolds with dimension $d_i$ and ${\rm Ric}(h_i) = \lambda_i h_i, \lambda_i > 0$.
Assume further that $M_1 = S^{k+1}$ and $h_1$ is the constant curvature $1$ metric.
Then the Lyapunov function $\mathscr F$ satisfies
$$ {\mathscr F} \geq n \prod_{i=1}^m  \ \lambda_i^{d_i/n}. $$
Equality holds for a trajectory of the GRS equations iff $g_i/\sqrt{\lambda_i}$ is
independent of $i \  (1 \leq i \leq m)$ and the traceless shape operator $L^{(0)}$ is $0$.
\end{prop}
\begin{proof} By (\ref{def-F}), we have
\begin{eqnarray*}
  {\mathscr F}  & \geq & v^{\frac{2}{n}}\ S = \left(\prod_{i=1}^m g_i^{\frac{2d_i}{n}} \right)
    \sum_{j=1}^m \ \frac{d_j\lambda_j}{g_j^2}  \\
    &\geq &  \left(\prod_{i=1}^m g_i^{\frac{2d_i}{n}}\right)
    n \prod_{i=1}^m \left(\frac{\lambda_i}{g_i^2} \right)^{\frac{d_i}{n}}  \\
    &=& n \prod_{i=1}^m  \ \lambda_i^{d_i/n},
\end{eqnarray*}
where we used the inequality (cf \cite{HLP}, p. 17)
$$ \prod_{i=1}^m a_i^{\frac{\rho_i}{\sum  \rho_j}} \leq \frac{\sum \rho_j a_j}{\sum \rho_j} $$
for $a_i, \rho_i \geq 0$. Equality holds in this inequality iff all the $a_i$ are equal,
which translates into our condition since we took $a_j = \lambda_j/g_j^2$ and $\rho_j= d_j$.

\end{proof}

Let us consider the $3$-dimensional subset of phase space
$$ \mathscr D :=\{ (g, L, \dot{u}): L^{(0)} = 0, \ g_i/\sqrt{\lambda_i} \ \mbox{\rm independent of } i \}. $$
By Proposition \ref{Fbound}, $\mathscr D$ is the set on which the function
$\mathscr F$ attains its (global) minimum value. It is invariant under the flow of
the GRS equations. In fact, if we parametrize this set by $x= g_i \sqrt{\frac{n-1}{\lambda_i}},
y=\dot{x}/x,$ and  $z=\dot{u}$, then the induced flow on $\mathscr D$ is
precisely the GRS equations for the special case $m=1, d_1 = n,$ which was analysed
by Kotschwar \cite{Kot}. Recall that Kotschwar showed that only two trajectories
in $\mathscr D$ represented complete smooth solitons with at least one singular
orbit (actually a point). However, when $m>1$, no trajectory lying in $\mathscr D$
represents a smooth soliton. When $m>1$, the non-compact solution acquires a conical
singularity and will be referred to as the {\em conical} Gaussian soliton. The compact
solution has two conical singularities and will be called the {\em spherical Einstein cone}.
On the other hand, there is a smooth non-compact complete solution which is the
product of the shrinking Gaussian soliton on $\R^{d_1+1}$ with the remaining
Einstein factors. (This does not represent a trajectory in $\mathscr D$).
We will call this solution the {\em smooth Gaussian soliton}, which is {\em rigid}
in the terminology of Petersen-Wylie \cite{PW}. For comparison purposes we list below
some information about these special solutions.

\begin{example} \label{SG} ({\em smooth Gaussian} on $\R^{d_1+1} \times M_2 \times \cdots \times M_m$)
The metric is given by
$$dt^2 + t^2 h_1 + \left(\frac{2\lambda_2}{-\epsilon}\right) \ h_2
    + \cdots + \left(\frac{2\lambda_m}{-\epsilon}\right) \ h_m$$
and $u(t) = -\frac{\epsilon}{4}t^2$. We further have
$$ \tr L = \frac{d_1}{t},  \ \
     \xi = \frac{\epsilon}{2}t + \frac{d_1}{t}, \ \
     {\mathcal E} =  -\frac{\epsilon^2}{4}t^2 -\frac{\epsilon}{2}(d_1 + 1). $$
\end{example}
\begin{example} ({\em conical Gaussian}) \label{CG}
The metric is given by
$$ dt^2 + t^2 \left( \frac{\lambda_1}{n-1}\ h_1 + \cdots + \frac{\lambda_m}{n-1}\ h_m   \right)$$
and $u(t) = -\frac{\epsilon}{4}t^2$. A conical singularity occurs at $t=0$ except when
$m=1$. The metric is Ricci-flat and complete at infinity. We further have
$$  \tr L = \frac{n}{t}, \ \
      \xi = \frac{\epsilon}{2}t + \frac{n}{t}, \ \
      {\mathcal E} = -\frac{\epsilon^2}{4} t^2 - \frac{\epsilon}{2}(n+1).   $$
\end{example}
\begin{example} ({\em spherical Einstein cone}) \label{SEC}
The metric, which is Einstein, is given by
$$ dt^2 + \left(\frac{\sin^2(\alpha t)}{{\alpha^2}(n-1)}\right)
     \left(\lambda_1  h_1 + \cdots + \lambda_m  h_m \right) $$
where $\alpha=\sqrt{-\epsilon/2n}$. We also have $ \xi = \tr L = \alpha n \cot(\alpha t). $
When $m>1$ there are conical singularities at $t=0$ and $t=\pi/\alpha$.
\end{example}

\bigskip
Let us now further specialise to the case $m=2$. This is of particular importance
because it includes some of the examples for which B\"ohm \cite{B1} was able to
produce infinite families of Einstein metrics.

We have found, in \S 1, a general winding number for the flow, representing winding
round the submanifold of Einstein trajectories. We shall now, in the case of warped
products on $2$ factors, investigate another winding number, similar to the one used
by B\"ohm in the Einstein case (cf \cite{B1}). This
counts winding around the subset $\mathscr D$ of phase space, which we introduced above.
We saw that trajectories in $\mathscr D$ represent solutions
which are equivalent to
warped products on {\em one} factor.

For convenience, we shall take the
Einstein constants $\lambda_i$ to be $d_i - 1$.
We recall that $\mathscr D$ is defined by
\[
\frac{X_1}{\sqrt{d_1}} = \frac{X_2}{\sqrt{d_2}} \;\; : \;\;
\frac{Y_1}{\sqrt{d_1}} = \frac{Y_2}{\sqrt{d_2}}.
\]
Note in particular that the intersection of $\mathscr D $ with the region
$\{\Ha=1, \Qa=0\}$ (cf. Remark \ref{HQ}) consists (with $\epsilon$ normalised to $-2n$)
of the curve parametrised by
$$X_i = \frac{\sqrt{d_i}}{n},\ \  Y_i = \frac{\sqrt{d_i} \sqrt{n-1}}{n} \sec t,
\ \ W= \frac{1}{n} \tan t.$$
This is the spherical Einstein cone solution in (\ref{SEC}), with $\alpha =1$.

\medskip

The flow-invariant subvariety $\mathscr Z$ of phase space given by $\Ha=1, \Qa=0, W=0$
is transverse to $\mathscr D$. More precisely, this subvariety meets $\mathscr D$ only
at the points $P_{\pm}$ given by
\[
X_i = \frac{\sqrt{d_i}}{n} \;\; :
Y_i = \pm \frac{\sqrt{d_i} \sqrt{n-1}}{n}.
\]
 Note that $P_{+}$ is the initial point (i.e. the value at $t=0$)
of the spherical cone and Gaussian soliton solutions. In fact
$P_{+}$ is a critical point of our equations (recall $t=0$ corresponds
to $s=-\infty$).

Let us now linearise our equations around $P_{+}$ in the invariant
subvariety $\mathscr Z$. This means we only consider vectors tangent to
$\mathscr Z$, so we have
\[
w=0 : \sqrt{d_1} x_1 + \sqrt{d_2} x_2 =0 = \sqrt{d_1} y_1 + \sqrt{d_2} y_2
\]
(the final equation follows from the condition $d\Qa=0$). Parametrising the
tangent space by $x_1, y_1$, we obtain the system
\[
\left( \begin{array}{c}
x_1 \\
y_1
\end{array} \right)^{\prime}=
\left( \begin{array}{cc}
\frac{1-n}{n} & \frac{2 \sqrt{n-1}}{n} \\
-\frac{\sqrt{n-1}}{n} & 0
\end{array}
\right)
\left( \begin{array}{c}
x_1 \\
y_1
\end{array} \right).
\]
The eigenvalues are roots of the quadratic
\[
\lambda^2 + \frac{n-1}{n} \lambda + \frac{2(n-1)}{n^2}
\]
which has discriminant $\frac{(n-1)(n-9)}{n^2}$.
We deduce
\begin{prop}
$P_{+}$ is a focus of the flow on the invariant variety $\mathscr Z$ transverse
to $\mathscr D$ if and only if $2 \leq n \leq 8$.
\end{prop}

So this winding number behaves in just the same way as its analogue in
the Einstein case in \cite{B1}, where it was used to produce symmetric metrics
(cf. remarks after equation (\ref{Weqn})).
However, to produce soliton analogues
of these B\"ohm metrics one would need to control $\dot{u}$ as well, probably
by means of the winding number of \S 1, and as we have seen this has quite
different behaviour. In the next section we shall investigate these and other
examples of B\"ohm numerically in the soliton case.

\section{\bf Numerical investigations}

The  non-existence results of \S 3 rule out certain special principal
orbit types but leave the existence problem open for most other orbit types.
In particular they do not apply to the examples considered by B\"ohm in \cite{B1},
who produced infinite families of cohomogeneity one Einstein metrics on certain
manifolds with low dimensions ($5 \leq n+1 \leq 9$).

In this section we report on some attempts to investigate the existence of
cohomogeneity one shrinking Ricci solitons using numerical methods, focusing
on the B\"ohm spaces and some related examples. The numerical methods employed
are relatively simple, but do show some interesting results.

We first consider an example that {\em is} known to admit a non-trivial
(i.e. non-Einstein) shrinking Ricci soliton, namely
 $\mathbb{CP}^{2}\sharp \,\overline{\mathbb{CP}}^{2}$.
The soliton in this case is a $U(2)$-invariant K\"ahler metric that was
constructed independently by Koiso \cite{Koi} and Cao \cite{Ca1}.
This metric provides an important test case for any ``soliton-hunting" programs,
and we shall henceforth refer to it as the Koiso-Cao soliton.

For the more complicated spaces considered by B\"ohm we are able to replicate
his Einstein metrics from the numerical search, but do not find any examples
admitting non-trivial shrinking solitons. We also examine some examples of
B\"ohm type but with dimension above the range where his Einstein existence
results apply. In this case too we do not find any solitons. The results seem
to hint that the cohomogeneity one soliton equations in the shrinking case exhibit
a high degee of rigidity and instability.

Let us now turn to the analysis of the equations. We are interested in solutions
to equation (\ref{gradRS}). Following B\"ohm  \cite{B1}, we consider situations
where the space of invariant metrics on the hypersurface is $2$-dimensional.
The cohomogeneity one metric therefore involves two functions, and the soliton
equations, being a system of second-order equations in the metric and the soliton
potential, are a dynamical system in $\mathbb{R}^6$. As remarked in \S 1, the
conservation law (\ref{cons1}) may be viewed as a constraint and  plays a vital role.

More precisely, we can consider a principal orbit $G/K$ and take one of
the special orbits to be $G/H$, where
$$ \g = \kf \oplus \p_1 \oplus \p_2$$
with $\p_1, \p_2$ being inequivalent irreducible $K$-modules and $\h = \kf \oplus \p_1$.
As in \S 4, we denote by $d_i$ the dimension of $\p_i$. To ensure smoothness at the
special orbit we need $H/K$ to be a sphere $\mathbb{S}^{d_1}$.
Note also that $Q:=G/H$ is isotropy irreducible and hence Einstein;
we shall denote its Einstein constant by $C_{Q}>0$. Our cohomogeneity
one metric is now given, with respect to an invariant background
inner product $B$ on $\p_1 \oplus \p_2$, by
\[
dt^2 + g_t = dt^2 + f(t)^2 B|_{\p_1} + h(t)^2 B|_{\p_2}.
\]
Note that the fibration
$H/K \rightarrow G/K \rightarrow G/H$  now becomes a Riemannian submersion
with respect to $g_t$ on $G/K$ and the metrics given by
$f(t)^2 B|_{\p_1}$ and $h(t)^2 B|_{\p_2}$ on $H/K$ and $G/H$ respectively.

Examples of this situation include the case when $G/K$ is a product
of two isotropy irreducible spaces with one being a sphere. Another case, familiar
in the Einstein situation from the work of B\'erard Bergery \cite{BB}, is when $d_1=1$, so $G/K$
is a circle bundle over $G/H$.

As usual, the same equations arise in certain more general situations which are not
strictly of cohomogeneity one type. For example in the B\'erard Bergery situation,
the hypersurface can be a circle bundle over an arbitrary Fano K\"ahler-Einstein base,
not necessarily homogeneous.

The Ricci soliton equations (\ref{soleqN}) and (\ref{soleqS}) specialize to:
\begin{equation}
-d_{1}\frac{\ddot{f}}{f}-d_{2}\frac{\ddot{h}}{h}+\ddot{u} +\frac{\epsilon}{2} = 0,
\end{equation}
\begin{equation} \label{feqn}
-\frac{\ddot{f}}{f}+(1-d_{1})\frac{\dot{f}^{2}}{f^{2}}-d_{2}\frac{\dot{f}\dot{h}}{fh}
     +\frac{d_{1}-1}{f^{2}} +\frac{d_{2}}{d_{1}}\|\mathcal{A}\|^{2}\frac{f^{2}}{h^{4}}
       + \frac{\dot{u}\dot{f}}{f}+\frac{\epsilon}{2} = 0,
\end{equation}
\begin{equation} \label{heqn}
-\frac{\ddot{h}}{h} +(1-d_{2})\frac{\dot{h}^{2}}{h^{2}}-d_{1}\frac{\dot{f}\dot{h}}{fh}
      +\frac{C_{Q}}{h^{2}} -2\|\mathcal{A}\|^{2}\frac{f^{2}}{h^{4}}
        +\frac{\dot{u}\dot{h}}{h} + \frac{\epsilon}{2} = 0,
\end{equation}
where the quantity $\|\mathcal{A}\|^{2}$ is the norm-squared of the O'Neill tensor
for the Riemanian submersion $H/K \rightarrow G/K \rightarrow G/H$ mentioned above.
It is a constant that depends only on the topology of the bundle.
When $G/K$ is a product of the two factors $H/K = \mathbb{S}^{d_1}$ and
$G/H$ then $\|\mathcal A \|^2$ is zero. We recall the conservation law (\ref{cons1}),
given in this setting by
\begin{equation}
\ddot{u}+\dot{u}\left(d_{1}\frac{\dot{f}}{f}+d_{2}\frac{\dot{h}}{h}\right)
    - \dot{u}^{2} - \epsilon u = C. \label{conseqn}
\end{equation}
Note that by changing the value of $u(0)$ we may fix the constant $C$ to be $0$,
and by a homothety of the metric we may alter the value of $\epsilon$.

To carry out our numerical study we shall consider the system (\ref{feqn})-(\ref{conseqn})
and introduce new variables $(z_1, z_2, z_3, z_4, z_5, z_6):=(f, \dot{f}, h, \dot{h}, u, \dot{u})$.
We then obtain the system

\begin{eqnarray}
   \dot{z_1} &=& z_2,  \\
   \dot{z_2} &=& -(d_1 -1)\frac{z_2^2}{z_1} -d_2 \frac{z_2 z_4}{z_3} + \frac{d_1-1}{z_1}
      + \frac{d_2}{d_1}\,\|\mathcal A\|^2 \frac{z_1^3}{z_3^4}
            + z_2 z_6 +\frac{\epsilon}{2}z_1,  \\
   \dot{z_3} &=& z_4, \\
   \dot{z_4} &=& -(d_2 -1)\frac{z_4^2}{z_3} -d_1\frac{z_2 z_4}{z_1} + \frac{C_Q}{z_3}
                   -2\,\|\mathcal A\|^2 \frac{z_1^2}{z_3^3} +z_4 z_6 + \frac{\epsilon}{2} z_3, \\
   \dot{z_5} &=& z_6,  \\
   \dot{z_6} &=& -z_6 \left(d_1 \frac{z_2}{z_1} + d_2 \frac{z_4}{z_3}\right) + z_6^2
              + \epsilon z_5.
\end{eqnarray}

The algorithm used to solve the above system numerically is the
Runge-Kutta method. A good account of it can be found in \cite{But},
and we include a brief discussion here for completeness. Given an ODE
$\dot{y}(t) = F(t,y)$ with initial condition $y(0)=y_{0}$, the Runge-Kutta
sequence of approximations is given by:
\begin{eqnarray*}
y_{n+1} &=& y_{n}+\frac{1}{6}({\sf k}_{1}+2{\sf k}_{2}+2{\sf k}_{3}+{\sf k}_{4}),\\
t_{n+1} &=& t_{n}+{\sf h},
\end{eqnarray*}
where
\begin{eqnarray*}
{\sf k}_{1} &=& F(t_{n},y_{n}),\\
{\sf k}_{2} &= & F(t_{n}+\frac{1}{2}{\sf h},y_{n}+\frac{1}{2}{\sf hk}_{1}),\\
{\sf k}_{3} &=& F(t_{n}+\frac{1}{2}{\sf h},\frac{1}{2}{\sf hk}_{2}), \\
{\sf k}_{4} &=&  F(t_{n}+{\sf h},{\sf hk}_{3}).
\end{eqnarray*}
The error at each step is $O(|{\sf h}|^{5})$ which gives an accumulated error of
$O(|{\sf h}|^{4})$.

The smoothness conditions at the singular orbit require that $(z_1,z_2,z_3,z_4,z_5,z_6)
=(0,1, \bar{h},0,\bar{u},0)$ where $\bar{h}$ is an arbitrary positive constant
and $\bar{u}$ is an arbitrary constant $\geq -(n+1)/2$ (by (\ref{Ebound})).
Note that this means that the initial Runge-Kutta step is not defined. We get around this
problem using the standard trick of constructing a power-series solution about a
small neighbourhood and then running the algorithm from this solution.

\subsection{Numerical results for $\mathbb{CP}^{2}\sharp \,\overline{\mathbb{CP}}^{2}$}

The metrics we consider here are $U(2)$-invariant and the principal orbits are Berger spheres
$S^3 = U(2)/U(1)$, so
we have $d_{1} = 1$ and $Q =\mathbb{S}^{2}$ with $C_{Q} = 1$. The constant
$\|\mathcal{A}\|^{2} =1$.
As remarked in the introduction, there is a non-trivial K\"ahler-Ricci soliton on this
manifold due to Koiso and Cao. There is also a Hermitian, non-K\"ahler, Einstein metric due to
Page \cite{Pa}. The Page metric is constructed explicitly in
our variables in \cite{Kod}, and in order to
check our algorithm we use the same normalisation and so take $\epsilon = -7.46562$.

In order to define a smooth metric on the blowup of $\mathbb{CP}^2$, we need the
$S^1$ fibre of the Berger sphere to collapse at both ends.
We are looking for a $t=t_{sol}>0$ such that
$$(f(t_{sol}),\dot{f}(t_{sol}),h(t_{sol}),\dot{h}(t_{sol}),u(t_{sol}),\dot{u}(t_{sol}))=(0,-1,\tilde{h},0,\tilde{u},0)$$
for some $\tilde{h} > 0$ and some $\tilde{u} > -2$. We output initial conditions where
$$SOL(t) = f^{2}(t)+(\dot{f}(t)+1)^{2}+\dot{h}^{2}(t)+\dot{u}^{2}(t)$$
satisfies $SOL<0.005$ for some $t > 0$. The plot below indicates the output of the
program. The parameters searched are $0 \leq \bar{h} \leq  3$, $-2 \leq \bar{u} \leq 2$
and the stepsize is $0.005$.

\begin{center}

\includegraphics{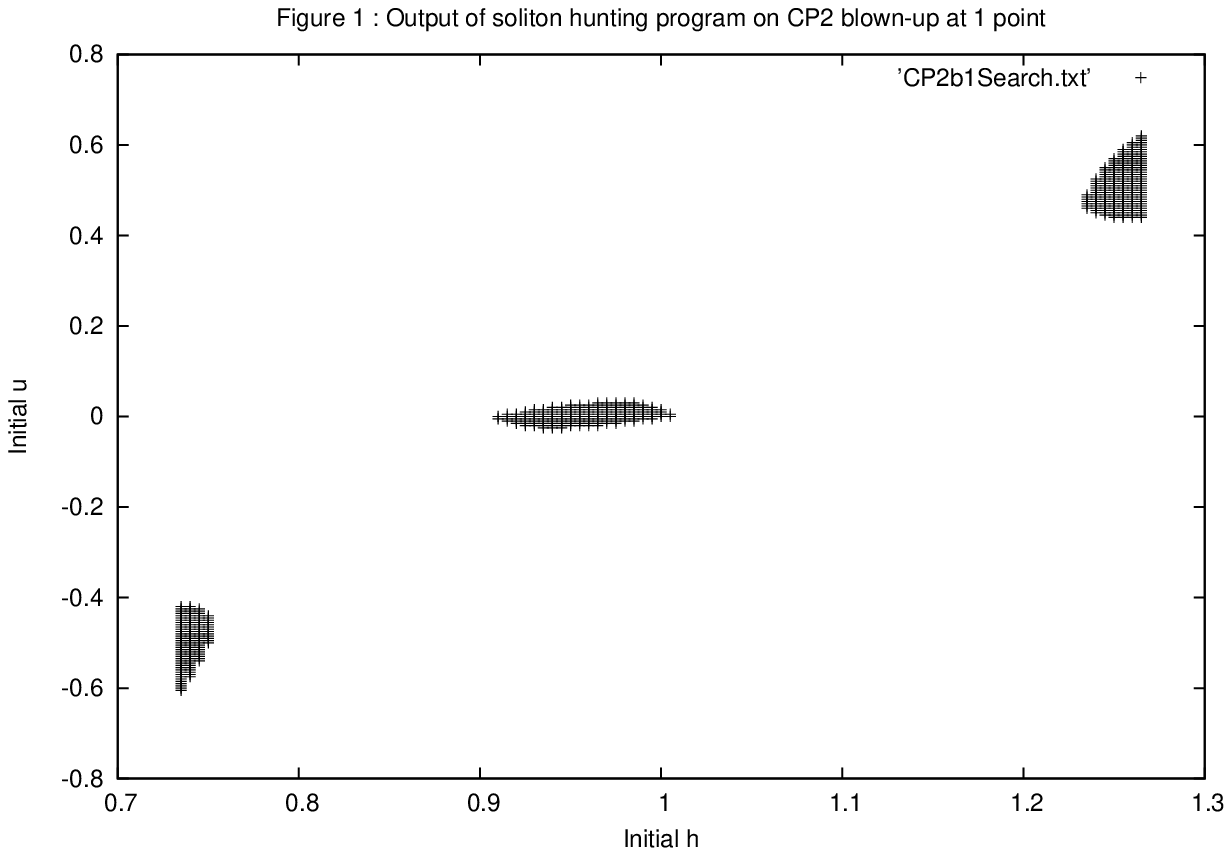}
\end{center}

The cluster about  $(\bar{h},\bar{u})=(0.7319,-0.5276)$ corresponds to the Koiso-Cao
soliton, the  cluster about $(\bar{h},\bar{u})=(0.9595,0)$ is the Page metric, and
the third cluster is the Koiso-Cao soliton with the conjugate complex structure.

We investigated for the soliton
the behaviour of the winding angle of \S 1 numerically, and
found that the angle decreases monotonically along the flow.

\subsection{Numerical results on $\mathbb{S}^{5}$}

Here $G = SO(3) \times SO(3)$ and the principal orbits are $\mathbb{S}^{2}\times \mathbb{S}^{2}$.
As we have mentioned before, B\"ohm \cite{B1} found infinitely many cohomogeneity
one Einstein metrics on this manifold. We choose the normalisation $\frac{\epsilon}{2}= -0.04$
so that the standard round metric has initial conditions $(h(0),u(0)) = (10,0).$

We need one of the $\mathbb{S}^2$ factors to collapse at one end, and the
other factor to collapse at the other end.
We are looking for a $t=t_{sol}>0$ such that
$$(f(t_{sol}),\dot{f}(t_{sol}),h(t_{sol}),\dot{h}(t_{sol}),u(t_{sol}),
   \dot{u}(t_{sol}))=(\tilde{f},0,0,-1,\tilde{u},0)$$
for some $\tilde{f} > 0$ and some $\tilde{u}$. We output initial conditions where
$$SOL(t) = \dot{f}^{2}(t)+h^{2}(t)+(\dot{h}(t)+1)^{2}+\dot{u}^{2}(t)$$
satisfies $SOL<0.005$ for some $t > 0$.

\begin{center}
\includegraphics{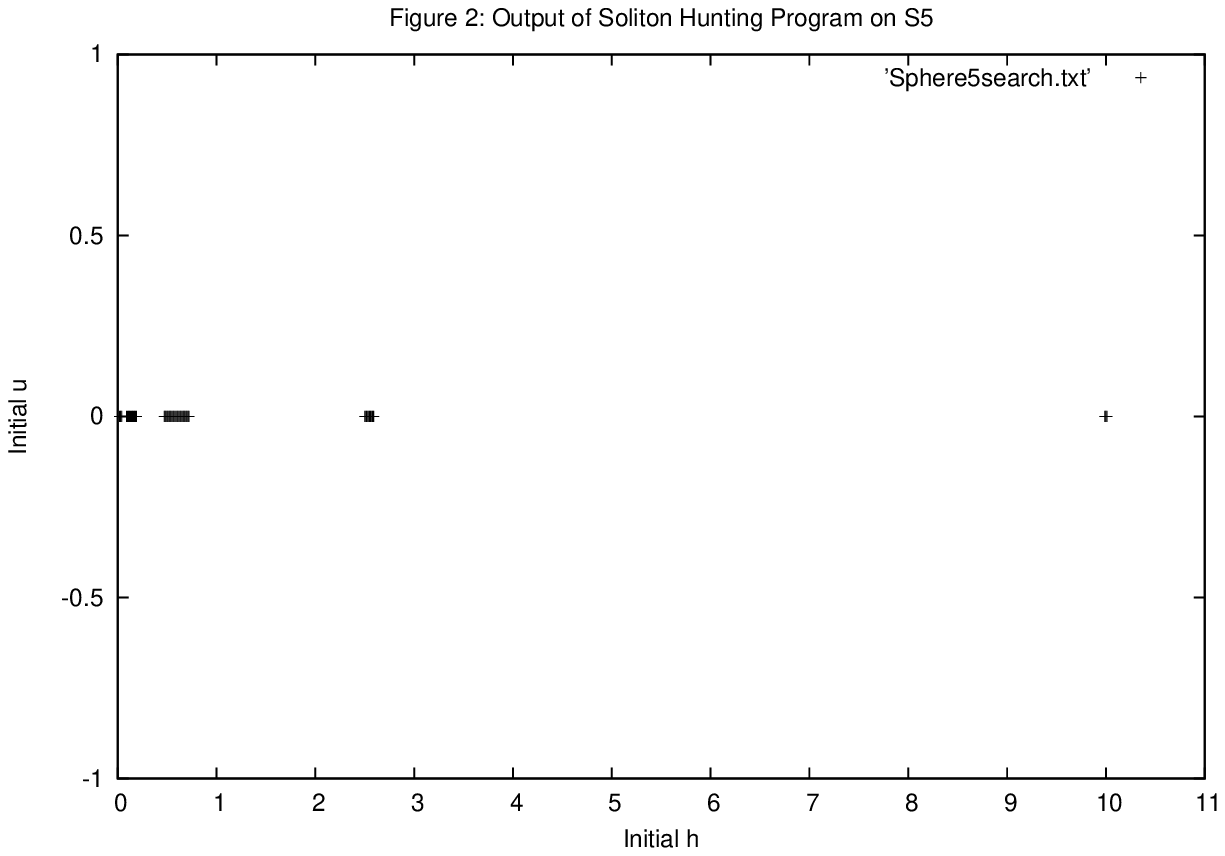}
\end{center}

The only values the program plotted are on the `Einstein axis', $\bar{u} = 0$.
We can see a cluster of points around $\bar{h} = 10$ which correspond to the
standard metric, a cluster around 2.5 which is the first B\"ohm metric ($\bar{h} \approx 2.53554 $),
then a final large cluster around 0 (of course there should be infinitely many
clusters between 2.5 and 0 but obviously our numerics cannot detect all these).
The authors are yet to find a cluster away from the Einstein axis.
It is interesting to restrict the algorithm to $\bar{u}=0$ and output the value of SOL
closest to $0$ for each $\bar{h}$. This plot gives us
\begin{center}
\includegraphics{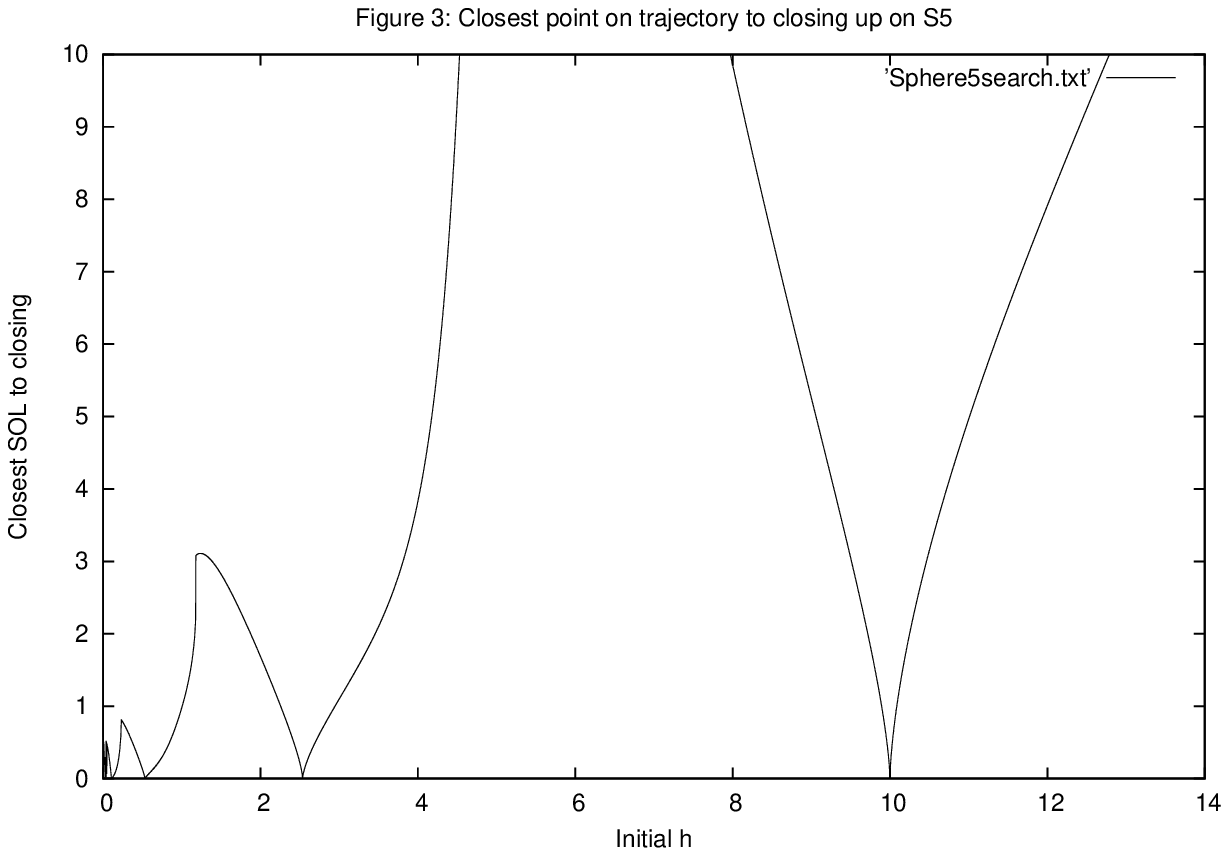}
\end{center}

We can clearly pick out the first three values of $\bar{h}$ corresponding to the
B\"ohm metrics $\bar{h} \approx 10, 2.5354, 0.53054$. We also see that the behaviour gets
increasingly complicated as $\bar{h}\rightarrow 0$.

\subsection{Numerical results on $\mathbb{S}^{2}\times\mathbb{S}^{3}$}

As in the $\mathbb{S}^{5}$ case, the group acting on the manifold is $G=SO(3)\times SO(3)$
and the principal orbits are $\mathbb{S}^{2}\times\mathbb{S}^{2}$.

We now need the same $\mathbb{S}^2$ factor to collapse at each end.
We are looking for a $t=t_{sol}>0$ such that
$$(f(t_{sol}),\dot{f}(t_{sol}),h(t_{sol}),\dot{h}(t_{sol}),u(t_{sol}),
    \dot{u}(t_{sol}))=(0,-1,\tilde{h},0,\tilde{u},0)$$
for $\tilde{h} > 0$ and some $\tilde{u}$.
As in the $\mathbb{S}^{5}$ case, we do not find any non-trivial solitons.
The hunting program does output clusters of points corresponding to the B\"ohm-Einstein
metrics on this manifold.  As it is more illuminating we output the smallest values of
SOL along some trajectories with $\bar{u}=0$.

\begin{center}
\includegraphics{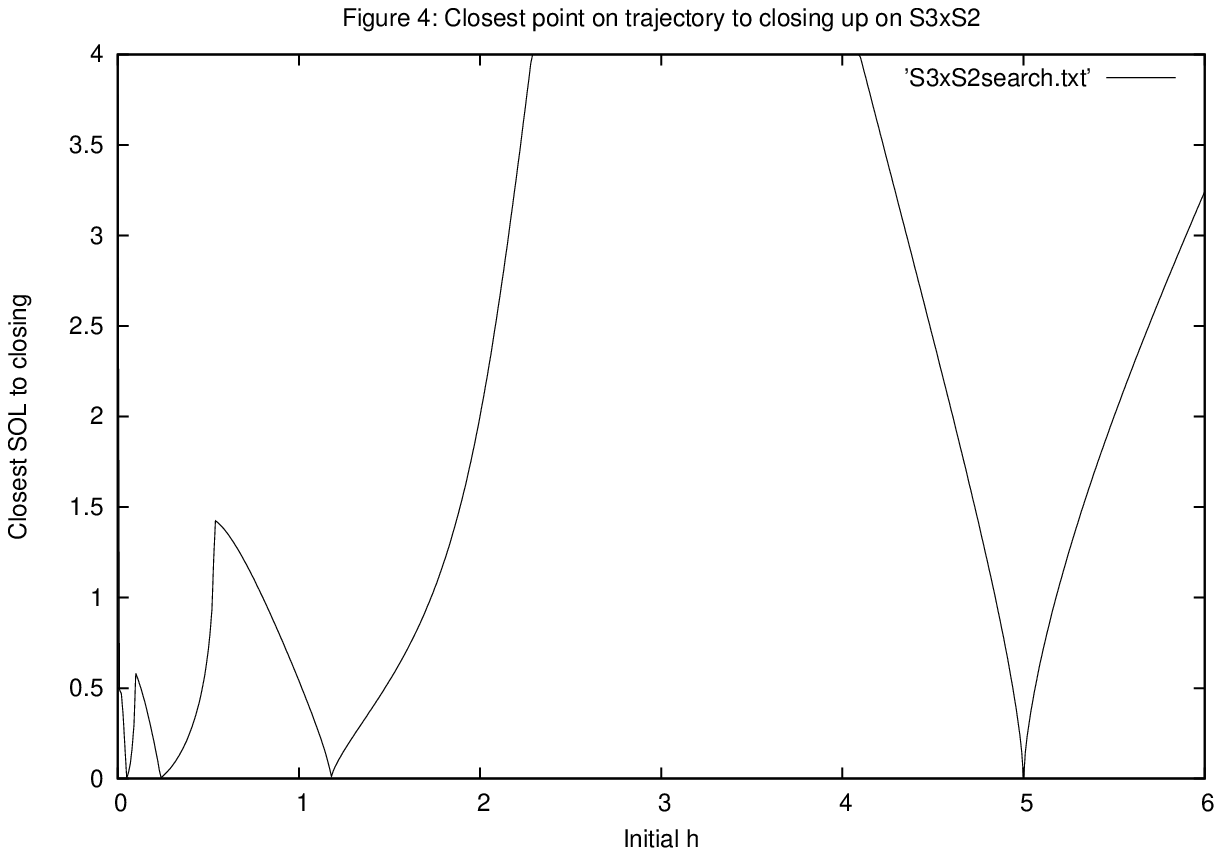}
\end{center}

So we see we can pick out the standard product metric on this space $\bar{h} = 5$ and
then the B\"ohm metrics corresponding to the initial values $\bar{h} \approx 1.1779$
and  $0.23571$.

\bigskip
As mentioned earlier, the analytical methods used by B\"ohm to construct
Einstein metrics only work in the dimension range $5 \leq \dim M \leq 9$.
It is not clear, however, for which orbit types this restriction is essential.

It is also conceivable that soliton solutions might exist on manifolds where
Einstein metrics do not. This is certainly the situation in the K\"ahler case.

We therefore examine next an example in higher dimensions, namely $\mathbb{S}^{11}$.

\subsection{Numerical results on $\mathbb{S}^{11}$}

In this example the group is $G=SO(6)\times SO(6)$ and the principal orbits
are $\mathbb{S}^{5}\times\mathbb{S}^{5}$.

The only cluster found by this search is one around $(10,0)$ corresponding to the
standard round metric on $\mathbb{S}^{11}$.

\subsection{Numerical results on $\HH\PP^{n+1} \sharp \,\,\overline{\HH\PP}^{n+1}$}

Another interesting case is that of the connected sum of two copies of
quaternionic projective space with opposite orientations.  In this case
the groups are $G = Sp(1) \times Sp(n+1)$, $K = \Delta Sp(1) \times Sp(n)$ and
$H = Sp(1) \times Sp(1) \times Sp(n)$. The principal orbits are $G/K  =\mathbb{S}^{4n+3}$,
and we have $d_{1}=3, d_{2}=4n$ and $C_{Q} = 4n+8$.  The special orbits
are $\HH\PP^{n}$, and the constant $\|\mathcal{A}\|^{2}=3$.

In the case $n=1$ B\"ohm \cite{B1} was able to prove the existence of a
cohomogeneity one Einstein metric on $\HH\PP^{2} \sharp \overline{\HH\PP}^{2}.$
He was also able to give numerical evidence for the existence of two cohomogeneity one
Einstein metrics in many cases where $n \geq 2$, although these examples are outside the
dimension range where his analytical methods work (some of these examples were also
found numerically by Page and Pope \cite{PP}).

 We also looked for solitons in these cases.
We ran the hunting program for a variety of values of $n$, using the initial values
$\bar{h} \in (0.1,10) $ and $\bar{u} \in (-\frac{4n+4}{2},\frac{4n+4}{2})$.
In order to increase accuracy, we used a step size of $0.001$ in the Runge-Kutta algorithm.

We could not find any non-trivial solitons.  When $n=1$ we found the B\"ohm metric
at $\bar{h} \approx 1.060793$. We include a plot of SOL below for this case. It demonstrates
a certain instability that all the Einstein metrics found on these manifolds shared,
namely that one needs to be close to the exact conditions in order to get the
trajectories to come close to closing up.

\begin{center}
\includegraphics{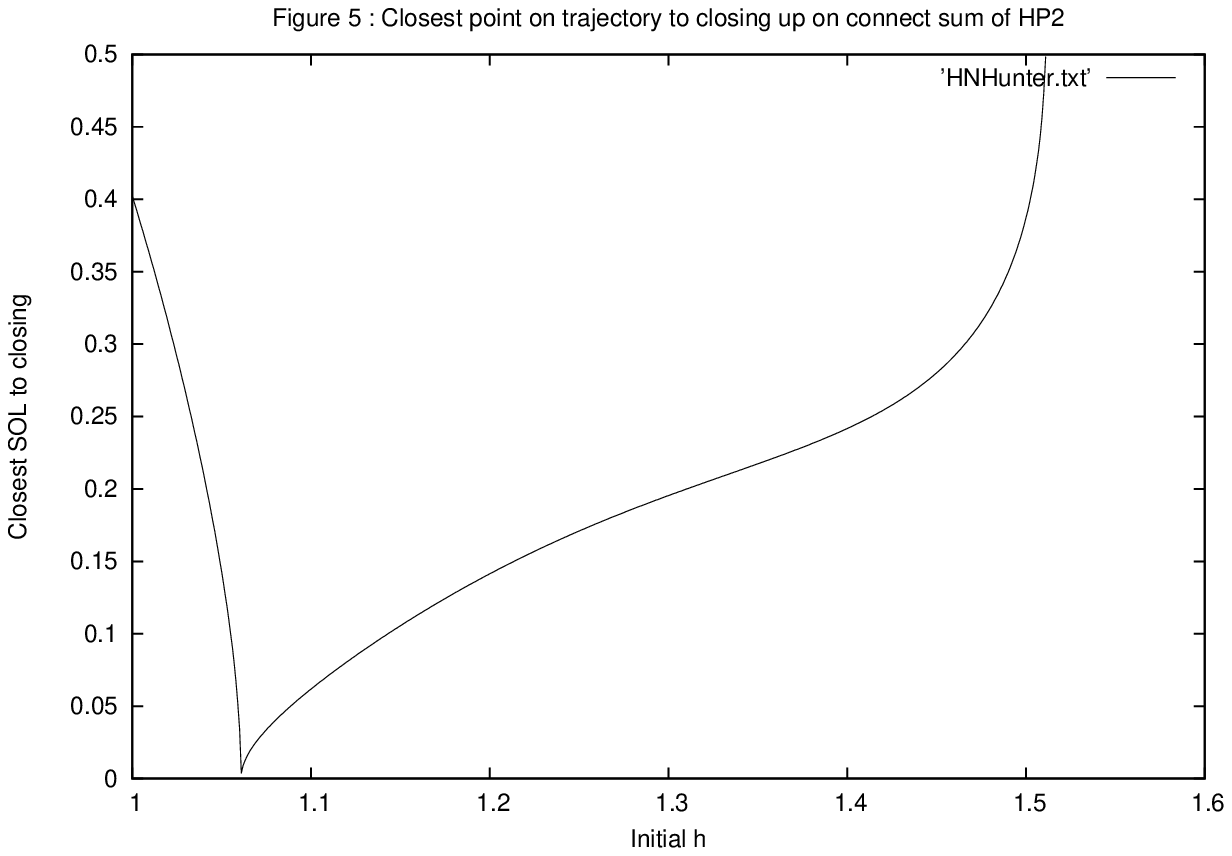}
\end{center}

We also include the same plot when $n=2$.  Notice we recover metrics that B\"ohm
found numerically with $\bar{h} \approx 1.0856062$ and $\bar{h}\approx 0.2184791$.

\begin{center}
\includegraphics{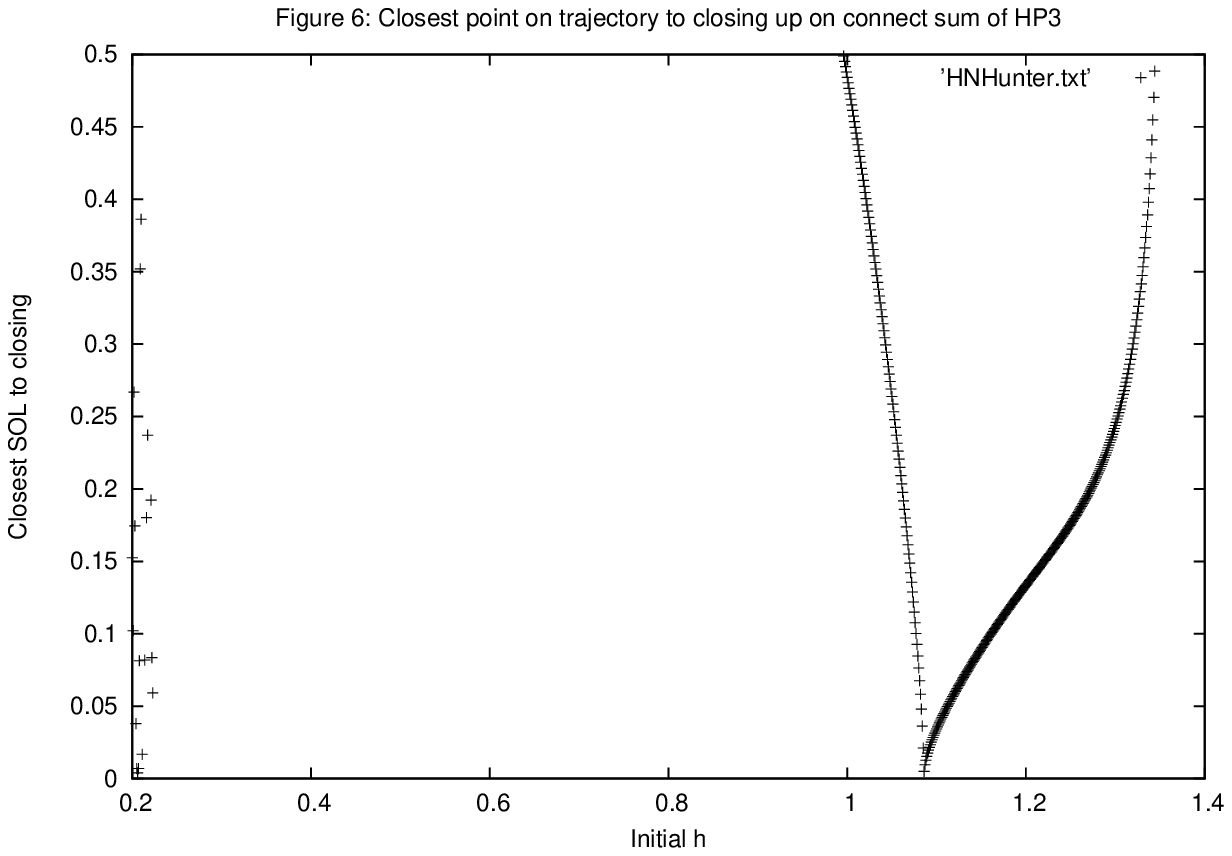}
\end{center}

\subsection{Numerical results for $F^{n+1}$ }

We also considered the space $F^{n+1}$ from \cite{B1}. For this space
$G = Sp(n+1), K=U(1) \times Sp(n)$ and $H = Sp(1) \times Sp(n)$, so the
principal orbit is $\mathbb{CP}^{2n+1}$, and the special orbits are $\HH \PP^{n}$.
The fibration $H/K \rightarrow G/K \rightarrow G/H$ is the twistor fibration for
$\HH \PP^{n}$, and we have $d_1=2$ and $d_2 = 4n$, as well as $C_{Q} =4n+8$
and $\|\mathcal{A}\|^{2}=8$.

Again we can recover Einstein metrics that B\"ohm and Gibbons-Page-Pope \cite{GPP}
found numerically, but we did not find any nontrivial solitons. We include below the results for $n=1$
where we can see there is evidence for an Einstein metric at $\bar{h} \approx 0.866$.
\begin{center}
\includegraphics{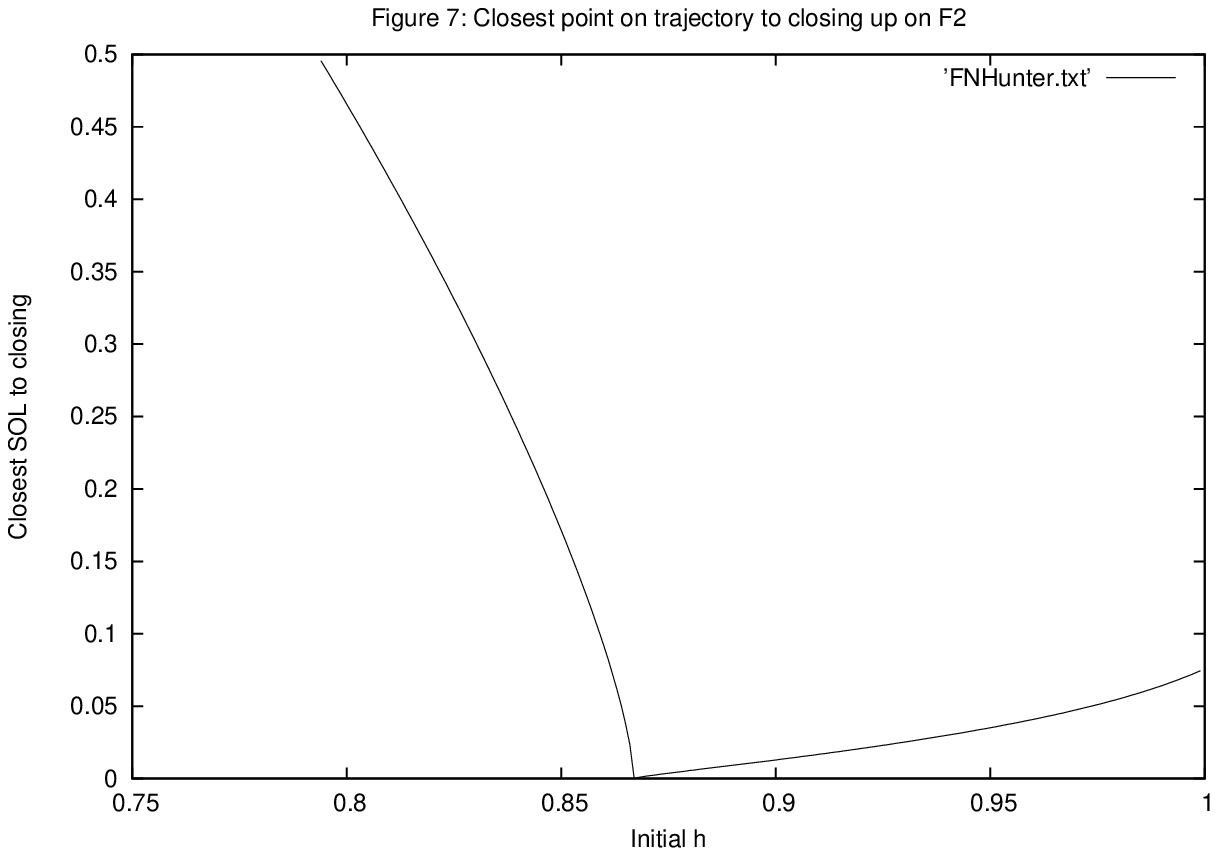}
\end{center}

\subsection{Numerical results on $Ca\PP^{2} \sharp \, \overline{Ca\PP}^{2}$}

Finally we also considered the connected sum of two copies of the Cayley
plane with opposite orientations. Now $G = Spin(9), K=Spin(7)$ and $H = Spin(8)$,
so the principal orbit is ${\mathbb S}^{15}$, and the special orbits are ${\mathbb S}^8$.
We have $d_1=7$ and $d_2 = 8$, with $Ric_{Q} = 28$ and $\|\mathcal{A}\|^{2}=7$.

We ran the soliton hunting program but did not find any evidence of Einstein
metrics or solitons.

\subsection{A Non-compact example - the Gaussian}
We conclude by discussing the smooth non compact Gaussian of Example
\ref{SG}.  We consider the case when the principal orbits are
$\mathbb{S}^{2}\times \mathbb{S}^{2}$. With our normalisations that
$C=0$ in equation (\ref{conseqn}) and $\frac{\epsilon}{2} = -4$, the
Gaussian soliton is described by
$$ dt^{2} +t^{2}dS_{2}^{2} + \frac{1}{2}dS_{2}^{2}$$
where $dS_{2}^{2}$ is the usual Einstein metric on $\mathbb{S}^{2}$
with Einstein constant 1. The soliton potential function is then given
by $u(t) = 2t^{2}-\frac{3}{2}.$ When we input the initial condition
$\bar{h} =0.5$ and $\bar{u} = -1.5$ and integrate we indeed recover
the Gaussian soliton for a short time.  However, as the numerical
errors accumulate along the trajectory the numerical solution deviates
from the Gaussian and very quickly becomes singular. Initial
conditions close to the Gaussian also display this behaviour.  This
suggests instability of the soliton, so that it might be difficult to
produce complete non-compact solutions close to the Gaussian. It should be noted
that there are some gap theorems for the Gaussian soliton on Euclidean space, due to
O. Munteanu-M.-T. Wang \cite{MuW} and Yokota \cite{Y}.

\end{document}